\documentclass[11pt,reqno]{amsart}
\pdfoutput=1 

\usepackage[english]{babel}
\usepackage[utf8]{inputenc}
\usepackage[T1]{fontenc}
\usepackage{lmodern} \normalfont 
\DeclareFontShape{T1}{lmr}{bx}{sc} { <-> ssub * cmr/bx/sc }{}
\usepackage{microtype}	

\usepackage{amssymb}
\usepackage{amsmath}
\usepackage{amsthm}
\usepackage{mathtools}
\usepackage{mathrsfs}
\usepackage{accents} 
\usepackage{etoolbox}
\usepackage{siunitx}
\usepackage{dsfont} 
\usepackage{graphicx}
\usepackage{color}
\usepackage[dvipsnames]{xcolor}
\usepackage{circuitikz}
\usepackage{tikz}
\usetikzlibrary{calc,positioning,shapes}
\usetikzlibrary{patterns,decorations.pathmorphing,decorations.markings}
\usepackage{pgfplots}
\pgfplotsset{colormap name=viridis,compat=newest}
\usepackage[margin=10pt,font=small,labelfont=bf,labelsep=endash]{caption}
\usepackage{subcaption}

\usepackage{booktabs}
\usepackage{paralist}
\usepackage{soul}

\numberwithin{equation}{section}

\usepackage{algorithm}
\usepackage{algorithmicx}
\usepackage{algpseudocode}

\usepackage{cite}
\usepackage{multirow} 
\usepackage{enumitem} 
\setlist[enumerate]{label=(\roman*)}
\usepackage{xspace}
\usepackage{listings}

\usepackage[colorlinks,linkcolor=teal,citecolor=teal,urlcolor=teal]{hyperref}
\usepackage[nameinlink,noabbrev]{cleveref}

\usepackage{pdflscape}

\theoremstyle{plain}
\newtheorem{theorem}{Theorem}[section]
\newtheorem{proposition}[theorem]{Proposition}
\newtheorem{lemma}[theorem]{Lemma}

\newtheorem{remark}[theorem]{Remark}

\newtheorem{assumption}[theorem]{Assumption}
\Crefname{assumption}{Assumption}{Assumptions}
\newtheorem{example}[theorem]{Example}

\newcommand{\N}{\ensuremath\mathbb{N}}
\newcommand{\Z}{\ensuremath\mathbb{Z}}

\newcommand{\R}{\ensuremath\mathbb{R}}

\newcommand{\spd}[1]{\mathbb{S}_{\succ}^{#1}}

\newcommand{\T}{\ensuremath\mathsf{T}}

\newcommand{\ds}{\,\mathrm{d}s}

\newcommand{\dzeta}{\,\mathrm{d}\zeta}

\newcommand{\ddx}{\ensuremath{\frac{\mathrm{d}}{\mathrm{d}x}}}
\newcommand{\ddt}{\ensuremath{\frac{\mathrm{d}}{\mathrm{d}t}}}
\newcommand{\tddt}{\ensuremath{\tfrac{\mathrm{d}}{\mathrm{d}t}}}

\newcommand{\hook}{\ensuremath{\hookrightarrow}}

\DeclareMathOperator{\id}{id}

\newcommand{\calC}{\mathcal{C}}

\newcommand{\calM}{\mathcal{M}}

\newcommand{\calX}{\mathcal{X}}

\newcommand{\cA}{\ensuremath{\mathcal{A}}}
\newcommand{\cB}{\ensuremath{\mathcal{B}}}
\newcommand{\cC}{\ensuremath{\mathcal{C}}}
\newcommand{\cD}{\ensuremath{\mathcal{D}}}

\newcommand{\cH}{\ensuremath{\mathcal{H}}}

\newcommand{\cM}{\ensuremath{\mathcal{M}}}

\newcommand{\cO}{\ensuremath{\mathcal{O}}}

\newcommand{\cQ}{\ensuremath{\mathcal{Q}}}

\newcommand{\cV}{\ensuremath{\mathcal{V}}}

\newcommand{\cX}{\ensuremath{\mathcal{X}}}

\newcommand{\cHV}{\ensuremath{\cH_{\scalebox{.5}{\cV}}}}
\newcommand{\cHQ}{\ensuremath{\cH_{\scalebox{.5}{\cQ}}}}

\newcommand{\cHQdual}{\ensuremath{\cH^*_{\scalebox{.5}{\cQ}}}}

\newcommand{\lineWidth}{1.2pt}

\definecolor{color0}{rgb}{1.0, 0.0, 0.0}
\definecolor{color1}{rgb}{0.0, 0.0, 1.0}
\definecolor{color2}{rgb}{0.1, 0.2, 0.9}
\definecolor{color3}{rgb}{0.9, 0.2, 0.1}
\definecolor{color4}{rgb}{0.8, 0.2, 0.8}
\definecolor{color5}{rgb}{0,0,0}

\definecolor{rabred}{rgb}{1.0, 0.0, 0.0}
\definecolor{rabblue}{rgb}{0.0, 0.0, 1.0}
\definecolor{rabcol}{rgb}{0.1, 0.2, 0.9}

\definecolor{crimson2143940}{RGB}{214,39,40}
\definecolor{darkgray176}{RGB}{176,176,176}
\definecolor{darkorange25512714}{RGB}{255,127,14}
\definecolor{forestgreen4416044}{RGB}{44,160,44}
\definecolor{gray127}{RGB}{127,127,127}
\definecolor{mediumpurple148103189}{RGB}{148,103,189}
\definecolor{orchid227119194}{RGB}{227,119,194}
\definecolor{sienna1408675}{RGB}{140,86,75}
\definecolor{steelblue31119180}{RGB}{31,119,180}

\definecolor{cbsblue}{RGB}{68,119,170}
\definecolor{cbscyan}{RGB}{102,204,238}
\definecolor{cbsgreen}{RGB}{34,136,51}
\definecolor{cbsyellow}{RGB}{204,187,68}
\definecolor{cbsred}{RGB}{238,102,119}
\definecolor{cbspurple}{RGB}{170,51,119}
\definecolor{cbsgrey}{RGB}{187,187,187}

\definecolor{mycolor1}{rgb}{0.00000,0.44700,0.74100}
\definecolor{mycolor2}{rgb}{0.85000,0.32500,0.09800}
\definecolor{mycolor3}{rgb}{0.92900,0.69400,0.12500}
\definecolor{mycolor4}{rgb}{0.46600,0.67400,0.18800}
\definecolor{mycolor5}{rgb}{0.49400,0.18400,0.55600}
\definecolor{mycolor6}{rgb}{0.16000,0.49000,0.49000}%
\definecolor{mycolor7}{rgb}{0.10000,0.49000,0.80000}%
\definecolor{mycolor8}{rgb}{0.06640,0.46484,0.19921} 
\definecolor{mycolor9}{rgb}{0.59765,0.59765,0.19921} 

\pgfplotscreateplotcyclelist{summationsol}{%
line width = \lineWidth, color=mycolor1\\
line width = \lineWidth, color=mycolor2\\
line width = \lineWidth, color=mycolor3\\
line width = \lineWidth, color=mycolor4\\
line width = \lineWidth, color=mycolor5\\
line width = \lineWidth, color=mycolor6\\
line width = \lineWidth, color=mycolor9\\
}

\pgfplotscreateplotcyclelist{toomuchdata}{%
semithick, color=forestgreen4416044\\
semithick, color=darkorange25512714\\
semithick, color=gray127\\
semithick, color=mediumpurple148103189\\
semithick, color=orchid227119194\\
semithick, color=sienna1408675\\
semithick, color=steelblue31119180\\
semithick, color=crimson2143940\\
}

\pgfplotscreateplotcyclelist{errresnorm}{%
thick, color=forestgreen4416044\\
thick, color=darkorange25512714\\
thick, color=gray127\\
semithick, color=forestgreen4416044\\
semithick, color=darkorange25512714\\
semithick, color=gray127\\
}

\pgfplotscreateplotcyclelist{rab}{%
dotted, color=color0, every mark/.append style={scale=2, solid, fill=color0}, mark=o\\
dashdotted, color=color0, every mark/.append style={scale=2, solid, fill=color1}, mark=o\\
solid, color=color0, every mark/.append style={scale=2, solid, fill=gray},mark=o\\
dotted, color=color1, every mark/.append style={scale=2, solid, fill=color0},mark=triangle\\
dashdotted, color=color1, every mark/.append style={scale=2, solid, fill=color1},mark=triangle\\
solid, color=color1, every mark/.append style={scale=2, solid, fill=gray},mark=triangle\\
dotted, every mark/.append style={solid, fill=gray},mark=*\\
dashdotted, every mark/.append style={solid, fill=gray},mark=square*\\
solid, every mark/.append style={solid, fill=gray}, mark=triangle*\\
}

\pgfplotscreateplotcyclelist{methodcompare}{%
color=cbsblue, every mark/.append style={scale=1.8, solid, fill=cbsblue}, very thick, mark=o\\
color=cbsblue, every mark/.append style={scale=2.2, solid, fill=cbsblue}, very thick, mark=diamond\\
color=cbsblue, every mark/.append style={scale=2, solid, fill=cbsblue}, very thick, mark=triangle\\
densely dashed, color=cbsred, every mark/.append style={scale=1.8, solid, fill=cbsred}, very thick, mark=o\\
densely dashed, color=cbsred, every mark/.append style={scale=2.2, solid, fill=cbsred}, very thick, mark=diamond\\
densely dashed , color=cbsred, every mark/.append style={scale=2, solid, fill=cbsred}, very thick, mark=triangle\\
loosely dotted, every mark/.append style={solid, fill=gray}, mark=*\\
loosely dotted, every mark/.append style={scale=1.2, solid, fill=gray}, mark=diamond*\\
loosely dotted, every mark/.append style={scale=1.2, solid, fill=gray}, mark=triangle*\\
}

\pgfplotscreateplotcyclelist{iterativeconv}{%
color=mycolor1, every mark/.append style={scale=1.8, solid, fill=mycolor1}, thick, mark=o\\
color=mycolor1, every mark/.append style={scale=2.2, solid, fill=mycolor1}, thick, mark=diamond\\
color=mycolor1, every mark/.append style={scale=2, solid, fill=mycolor1}, thick, mark=triangle\\
loosely dashed, color=mycolor4, every mark/.append style={scale=1.8, solid, fill=mycolor4}, thick, mark=o\\
densely dashed, color=mycolor4, every mark/.append style={scale=2.2, solid, fill=mycolor4}, thick, mark=diamond\\
loosely dotted, color=mycolor4, every mark/.append style={scale=2, solid, fill=mycolor4}, thick, mark=triangle\\
loosely dashed, every mark/.append style={solid, fill=gray}, mark=*\\
densely dashed, every mark/.append style={scale=1.2, solid, fill=gray}, mark=diamond*\\
loosely dotted, every mark/.append style={scale=1.2, solid, fill=gray}, mark=triangle*\\
densely dotted, every mark/.append style={solid, fill=gray}, mark=square*\\
densely dashdotted,every mark/.append style={solid, fill=gray}, mark=pentagon*\\
}

\pgfplotscreateplotcyclelist{itertol}{%
loosely dashed, color=crimson2143940, every mark/.append style={scale=1.6, solid, fill=crimson2143940}, thick, mark=*\\
densely dashed, color=darkorange25512714, every mark/.append style={scale=2.0, solid, fill=darkorange25512714}, thick, mark=diamond*\\
loosely dotted, color=steelblue31119180, every mark/.append style={scale=1.9, solid, fill=steelblue31119180}, thick, mark=triangle*\\
densely dotted, color=forestgreen4416044, every mark/.append style={scale=1.3, solid, fill=forestgreen4416044}, thick, mark=square*\\
every mark/.append style={solid, fill=gray}\\
densely dashed, every mark/.append style={solid, fill=gray}, mark=triangle*
loosely dotted, every mark/.append style={solid, fill=gray}, mark=*\\%
densely dotted, every mark/.append style={solid, fill=gray}, mark=square*\\%
densely dashdotted,every mark/.append style={solid, fill=gray}, mark=pentagon*\\%
}

\definecolor{corrColor}{RGB}{60,124,155}

\newcommand{\abbr}[1]{#1\xspace}
\newcommand{\BDF}{\abbr{BDF}}

\newcommand{\DDE}{\abbr{DDE}}
\newcommand{\DDEs}{\abbr{DDEs}}

\newcommand{\shift}[2]{\Delta_{#1}{#2}}
\newcommand{\bdf}[2]{\Xi_{#1}{(#2)}}
\newcommand{\delayOperKdisc}[1]{\Psi_{#1}} 
\newcommand{\numDelay}{\delta}
\newcommand{\numBDF}{k}

\title[Decoupling multistep schemes for elliptic--parabolic problems]{Decoupling multistep schemes for\\ elliptic--parabolic problems}
\author{Robert Altmann${}^\dagger$ \and Abdullah Mujahid${}^{\star}$ \and Benjamin Unger${}^\star$}

\address{${}^{\dagger}$ Institute of Analysis and Numerics, Otto von Guericke University Magdeburg, Universit\"atsplatz 2, 39106 Magdeburg, Germany}
\email{robert.altmann@ovgu.de}
\address{${}^{\star}$ Stuttgart Center for Simulation Science (SC SimTech), University of Stuttgart, Universit\"{a}tsstr.~32, 70569 Stuttgart, Germany}
\email{\{abdullah.mujahid,benjamin.unger\}@simtech.uni-stuttgart.de}
\date{\today}

\begin{document}

\begin{abstract} 
We study the construction and convergence of decoupling multistep schemes of higher order using the backward differentiation formulae for an elliptic--parabolic problem, which includes multiple-network poroelasticity as a special case. These schemes were first introduced in [Altmann, Maier, Unger, BIT Numer.~Math., 64:20, 2024], where a convergence proof for the second-order case is presented. Here, we present a slightly modified version of these schemes using a different construction of related time delay systems. We present a novel convergence proof relying on concepts from G-stability applicable for any order and providing a sharper characterization of the required weak coupling condition. The key tool for the convergence analysis is the construction of a weighted norm enabling a telescoping argument for the sum of the errors. 
\end{abstract}

\maketitle
{\footnotesize \textsc{Keywords:} poroelasticity, decoupling, higher-order discretization, backward differentiation formulae}

{\footnotesize \textsc{AMS subject classification:} 65M12, 65J10, 76S05}


\section{Introduction}
\label{sec:intro}

This paper explores semi-explicit time integration schemes of higher order for the equations of poroelasticity\cite{Bio41} or, more generally, for linear elliptic--parabolic problems. Among important applications of this problem are biomechanics, where the human brain and heart are modeled as poroelastic medium with multiple fluid networks~\cite{VarCT+16,EliRT23}, and geomechanics, where poroelasticity serves as a model problem~\cite{Zob10}.

The well-posedness of the considered elliptic--parabolic problem is studied in~\cite{Sho00}. To reduce the computational effort resulting from the coupling, several decoupling strategies exist~\cite{MikW13,KimTJ11a,KimTJ11b,StoBKNR19,BotBNKR17}, most of which are based on a fixed-point iteration. However, higher-order methods in the iterative framework need many inner iteration steps and hence call for a stronger contraction\cite{AltMU24b}.

Alternatively, a semi-explicit method of first order, as proposed and analyzed in~\cite{AltMU21b}, decouples the system whenever a weak coupling condition is satisfied. The design of the semi-explicit method is based on constructing a related partial differential equation with a time delay, approximating the solution of the original system up to a desired order. An extension to second order is considered in~\cite{AltMU24} and to nonlinear problems in~\cite{AltM22}. Unlike the step size restriction in explicit methods, semi-explicit methods require a weak coupling condition, which restricts the problem class given by the material parameters. To lift this restriction, schemes with a fixed number of relaxation steps are introduced in~\cite{AltD24,AltD24ppt}.

In this article, we provide closed-form expressions for the construction of
the delay systems and develop the necessary tools to show the convergence of the semi-explicit schemes using the \emph{backward differentiation formulae} (\BDF). In the standard analysis of \BDF-$k$ methods, one can construct a weighted norm with a symmetric positive definite $k\times k$ matrix $G$. The weighting allows for setting up a telescoping sum of the errors and, hence, is the key tool for the convergence analysis. Moreover, it exploits the equivalence between A--stability and G--stability~\cite{Dah78}; see also~\cite[Ch.~V.~6]{HaiW96}.
For methods that are not A--stable, one constructs additional multipliers~\cite{NevO81}; see also~\cite[Ch.~V.~8]{HaiW96}.
In practice, explicitly constructing the $G$ matrix is often tricky. Hence, the classical way is to prove that the method is A--stable (potentially with appropriate multipliers) and then infer the existence of $G$. In this paper, we proceed differently and explicitly construct the weighted norm. Our main results are:
\begin{enumerate}
	\item We present a novel way to construct the related delay equation of a given order, which, in contrast to the approach presented in \cite{AltMU24}, is explicit and requires fewer delays. 
	\item We use the norm defined through $G$ (assuming that it exists) to present a novel convergence proof that works for any \BDF order and provides a sharper characterization of the weak coupling condition compared to~\cite{AltMU24}.
	\item We explicitly construct the weighting matrix~$G$ for convergence up to order three.
\end{enumerate}

The remainder of this paper is organized as follows. After this introduction, the abstract model problem is introduced in~\Cref{sec:prelim} with particular examples from poroelasticity.
The design of semi-explicit methods via delay systems and their well-posedness is discussed in~\Cref{sec:prelim:delay}. The main result of this article is \Cref{thm:BDFdelay}, where the convergence analysis of the decoupling multistep methods based on \BDF schemes is given. The tool to prove higher-order convergence is \Cref{ass:summation:bdf}, which is the basis for the telescoping argument used in the proof.
The evidence for this assumption is discussed thoroughly in~\Cref{sec:summation:bdf}. Finally, we verify the theoretical results numerically in~\Cref{sec:numerics}.
%
%
\subsection*{Notation}
Throughout the paper, we write~$a \lesssim b$ to indicate that there exists a
generic constant~$C > 0$, independent of spatial and temporal discretization
parameters, such that~$a \leq C b$.
Moreover, we denote the Bochner spaces on the time interval~$[0,T]$ for a
Banach space~$\calX$ by~$L^p(0,T;\calX)$
and $W^{k,p}(0,T;\calX)$, where $p\ge 1$, $k\in\N$.
The $L^2$ norm is denoted by $\Vert\cdot\Vert$. 
%
%
\section{Abstract Formulation and Examples}
\label{sec:prelim}
Consider the Hilbert spaces
\begin{equation*}
	\cV\vcentcolon=[H_{0}^{1}(\Omega)]^m, \qquad
	\cQ\vcentcolon=H_{0}^{1}(\Omega), \qquad
	\cHV\vcentcolon=[L^{2}(\Omega)]^m, \qquad 
	\cHQ\vcentcolon=L^{2}(\Omega),
\end{equation*}
resulting in the Gelfand triples $\cV \hook \cHV\simeq\cHV^{*}\hook\cV^{*}$
and $\cQ \hook \cHQ\simeq\cHQ^{*}\hook\cQ^{*}$, cf.~\cite[Sec.~23.4]{Zei90}. 
The system of interest is the following abstract coupled elliptic--parabolic
problem:
given sufficiently smooth source terms $f\colon[0,T] \to \cV^*$
and $g\colon[0,T] \to \cQ^*$, seek abstract functions
$u\colon [0,T]\rightarrow\cV$ and $p\colon [0,T]\rightarrow\cQ$ such that for 
almost every~$t \in (0,T]$ it holds that
\begin{subequations}
	\label{eq:ellpar}
	\begin{align}
		a(u,v) - d(v, p) 
		&= \langle f, v \rangle, \label{eq:ellpar:a} \\
		d(\dot u, q) + c(\dot p,q) + b(p,q) 
		&= \langle g, q\rangle \label{eq:ellpar:b} 
	\end{align}
	for all test functions~$v\in \cV$, $q \in \cQ$. 
	As initial data, we assume 
	\begin{align}
		u(0) = u^0 \in \cV, \qquad 
		p(0) = p^0 \in \cHQ \label{eq:ellpar:c} 
	\end{align}
\end{subequations}
to be \emph{consistent}, i.e.,
we assume $a(u^0,v) - d(v, p^0) = \langle f(0),v\rangle$ for all~$v\in \cV$. 
The bilinear forms~$a\colon \cV\times\cV \to \R$,
$b\colon \cQ\times\cQ\to\R$, and $c\colon \cHQ\times\cHQ\to\R$
are assumed to be symmetric, continuous, and elliptic in their respective spaces. 
The corresponding ellipticity and continuity constants of $\mathfrak{a}\in\{a, b, c\}$ are denoted by
$c_{\mathfrak{a}}, C_{\mathfrak{a}} > 0$, respectively.
For notational convenience we introduce the 
$b$-norm $\|\cdot\|_{b} \vcentcolon= b(\cdot,\cdot)^{1/2}$ and the
$c$-norm $\|\cdot\|_{c} \vcentcolon= c(\cdot,\cdot)^{1/2}$ satisfying
\begin{align*}
	\tfrac{1}{C_b}\Vert \cdot \Vert^{2}_{b}
	\le \Vert \cdot \Vert_{\cQ}^{2}
	\le \tfrac{1}{c_b}\Vert \cdot \Vert^{2}_{b}
	\quad \text{and} \quad
	\tfrac{1}{C_c}\Vert \cdot \Vert^{2}_{c}
	\le \Vert \cdot \Vert_{\cHQ}^{2}
	\le \tfrac{1}{c_c}\Vert \cdot \Vert^{2}_{c},
\end{align*}
respectively.
Furthermore, $d\colon\cV\times\cHQ\to\R$ is bounded,
i.e., we assume that there exists a positive constant~$C_d > 0$ such
that~$d(u,p) \leq C_d \Vert u \Vert_{\cV} \Vert p\Vert_{\cHQ}$ for
all~$u \in \cV$, $p\in \cHQ$.  
Without the coupling term represented by the bilinear form $d$,
equation~\eqref{eq:ellpar:a} is elliptic in $u$ and equation~\eqref{eq:ellpar:b}
is parabolic in $p$.
This explains why system~\eqref{eq:ellpar} is called an elliptic--parabolic problem. 
At this point, we would like to emphasize that $c(\dot p,q)$ -- and similarly $d(\dot{u},q)$ -- should be understood as $\ddt c(p,q)$, see also \cite{Muj22}.
\begin{example}[linear poroelasticity]
Let~$\Omega \subseteq \R^m$, $m\in\{2,3\}$, be a bounded Lipschitz domain and $[0,T]$ a time interval with $T>0$. 
We consider the quasi-static Biot poroelasticity model with homogeneous Dirichlet boundary conditions as introduced in~\cite{Bio41}, 
where one seeks the deformation~$u\colon [0,T]\times\Omega\rightarrow\R^{m}$
and the pressure~$p\colon [0,T]\times\Omega\rightarrow\R$ satisfying
\begin{subequations}
	\label{eq:pdes}
	\begin{align}
		- \nabla\cdot\sigma(u) + \alpha \nabla p 
		&= {\hat f} \qquad\text{in }(0,T]\times\Omega, \label{eq:pdes:a}\\
		\partial_{t} \big(\alpha \nabla\cdot u + \tfrac{1}{M} p\big)
		-\nabla\cdot(\kappa\nabla p)
		&= {\hat g} \qquad\text{in } (0,T]\times\Omega. \label{eq:pdes:b}
	\end{align}
\end{subequations}
Here, $\sigma$ denotes the stress tensor
\begin{equation*}
	\sigma(u) = {\mu}\, \big(\nabla u + (\nabla u)^\T \big) 
	+ {\lambda}\, (\nabla \cdot u) \id
\end{equation*} 
with Lam\'{e} coefficients ${\lambda}$ and ${\mu}$.
The permeability is denoted with $\kappa$.
Moreover, $\alpha$ denotes the Biot--Willis fluid--solid coupling coefficient
and $M$ the Biot modulus.
The right-hand sides $\hat{f}$ and $\hat{g}$ are
the volumetric load and the fluid source terms, respectively,
modeling an injection or production process. 
We refer to \cite{AltMU21b} for the definition of the corresponding bilinear forms resulting in~\eqref{eq:ellpar}.  
\end{example}
\begin{example}[multiple network poroelasticity]
A generalization of the Biot poroelasticity model
is given by the multiple network poroelasticity theory, where we seek the
deformation~$u\colon [0,T]\times\Omega\rightarrow\R^{m}$ and pressure
variables~$p_{i}\colon [0,T]\times\Omega\rightarrow\R$ for $i = 1, \dots, J$
where $J \in \N$ denotes the number of fluid networks,
satisfying the coupled system 
\begin{align*}
	- \nabla\cdot\sigma(u) + \sum_{i=1}^{J}\alpha_i \nabla p_i 
	&= {\hat f} \qquad\text{in } (0,T]\times\Omega,\\
	\partial_{t} \Big(\alpha_i \nabla\cdot u
	+ \tfrac{1}{M_i} p_{i} \Big) - \nabla\cdot \big(\kappa_{i}\nabla p_i \big)
	+ \sum_{j\ne i}^{J}\beta_{ij}(p_i - p_j)
	&= {\hat g}_{i} \qquad\text{in } (0,T]\times\Omega.
\end{align*}
Given some additional assumptions on the exchange rates between the fluid networks $\beta_{ij}$, this model has a similar structure as~\eqref{eq:ellpar}. 
In practice, the exchange rates are usually symmetric and sufficiently small such that the ellipticity property of the bilinear form $b$ is satisfied~\cite{AltMU21c}. 
\end{example}
To describe the stability of the decoupling schemes introduced in this article, we define
\begin{equation}
	\label{eqn:couplingStrength}
	\omega\vcentcolon=\tfrac{C_d^{2}}{c_a c_c},
\end{equation}
as the \emph{coupling strength} between the elliptic and the parabolic equation. 
This value plays a crucial role for the convergence analysis and depends on the physical coefficients of the application. We refer to \cite[Tab.~4]{DetC93} for concrete values appearing in poroelasticity; see also \cite[Tab.~3.1]{AltMU24}. 

The abstract formulation~\eqref{eq:ellpar} can also be written in terms of operators in the respective dual spaces. For this, let $\cA\colon\cV\rightarrow\cV^*$, $\cB\colon\cQ\rightarrow\cQ^*$, $\cC\colon\cHQ\rightarrow\cHQdual$, 
and $\cD\colon\cV\rightarrow\cHQ^{*}$ be the operators associated with the bilinear forms $a$, $b$, $c$, and $d$ respectively. This then leads to the equivalent formulation 
\begin{subequations}
\label{eq:ellpar:opt}
    \begin{align}
        \cA u - \cD^{*} p 
        &= f \qquad \text{in } \cV^{*}, \label{eq:ellpar:opt:a}\\
        \cD{\dot u} + \cC{\dot p} + \cB p 
        &= g \hspace{0.79cm} \text{in } \cQ^{*}.\label{eq:ellpar:opt:b}
    \end{align}
\end{subequations}
Since $\cA$ is invertible, we can eliminate the variable $u$, leading to the parabolic equation
\begin{align}
	\label{eq:ppde}
   (\cM + \cC)\dot{p} + \cB p &= r,
\end{align}
where $r \vcentcolon= g - \cD\cA^{-1}\dot{f}$ and $\cM \vcentcolon= \cD\cA^{-1}\cD^{*}$ is a self-adjoint and non-negative operator.
%
%
\section{Related Delay Equations}
\label{sec:prelim:delay}

The key idea for the construction as well as convergence analysis of a decoupling time integration scheme in~\cite{AltMU21b} is to introduce an approximation of the coupled elliptic--parabolic problem~\eqref{eq:ellpar:opt} by adding a time delay $\tau>0$ in \eqref{eq:ellpar:opt:a}, where the time delay equals the time step size. In more detail, one considers the delay system 
\begin{subequations}
	\label{eq:ellpar:opt:delay}
	\begin{align}
	\cA \bar{u} - \cD^{*} \widehat{p}(t;\tau) 
	&= f \qquad \text{in } \cV^{*}, \label{eq:ellpar:opt:delay:a}\\
	\cD \dot{\bar{u}} + \cC{\dot {\bar p}} + \cB {\bar p} 
	&= g \qquad \text{in } \cQ^{*}, \label{eq:ellpar:opt:delay:b}
	\end{align}
\end{subequations}
where $\widehat{p}(t;\tau)$ is an approximation of $\bar{p}$ using the time delay $\tau$. In the first-order case studied in \cite{AltMU21b} we have $\widehat{p}(t;\tau) = \shift{\tau}{\bar{p}}(t)$ with the shift operator~$\shift{\tau}{p} \vcentcolon= p(\,\cdot-\tau)$. Let us emphasize that, in contrast to the original system~\eqref{eq:ellpar:opt}, such a delay system calls for a \emph{history function}, i.e., a prescription of $\bar{p}$ for $-\tau\leq t\leq 0$. 

If an implicit time discretization scheme of order $\numBDF$ with time step
size $\tau$ is used to discretize~\eqref{eq:ellpar:opt:delay}, then the
equations~\eqref{eq:ellpar:opt:delay:a} and \eqref{eq:ellpar:opt:delay:b} are
decoupled in the sense that one can first solve \eqref{eq:ellpar:opt:delay:a}
for~$\bar{u}$ and then~\eqref{eq:ellpar:opt:delay:b} for $\bar{p}$.
This was done and analyzed for the implicit Euler method in~\cite{AltMU21b} and for
\BDF-$2$ with a different $\hat p$ in~\cite{AltMU24}.
It goes without saying that a higher-order scheme is only reasonable if the delay
approximation~\eqref{eq:ellpar:opt:delay} is of the same order. 

\subsection{Higher-order approximation via Taylor series expansion}
For a higher-order approximation, a general time delayed approximation of $p(t)$ is
sought.
This means that we consider a delay term~${\widehat p}(t;\tau)$ with 
\begin{displaymath}
	p(t) 
	= {\widehat p}(t;\tau) + \cO(\tau^\numDelay).
\end{displaymath}
One straightforward possibility is to consider a Taylor series expansion of $p(t)$ around $t-\tau$.
This, however, leads to a \emph{delay differential equations} (\DDEs) of advanced type --- see \cite{BelC63} for a precise definition --- as shown in the following example, which is motivated from \cite[Ex.~2.1]{Ung18}. 
\begin{example}
	Consider two terms in the Taylor series, i.e., $\bar{p} \approx \shift{\tau}{{ {\bar p}}} + \tau\shift{\tau}{{{\dot {\bar p}}}}$. Solving~\eqref{eq:ellpar:opt:delay:a} for $\bar{u}$ and substituting this in \eqref{eq:ellpar:opt:delay:b} results the inherent delay differential equation
	\begin{align*}
	\cC {\dot {\bar p}} + \cB {\bar p} 
	= r - \cM(\shift{\tau}{{{\dot {\bar p}}}} + \tau\shift{\tau}{{{\ddot {\bar p}}}}).
	\end{align*}
	If we consider the scalar case where we replace all operators by the constant one and set $f \equiv 0$, $g \equiv 0$, then we obtain~${\dot {\bar y}} + {\bar y} = \shift{\tau}{{{\dot {\bar y}}}} + \tau\shift{\tau}{{{\ddot {\bar y}}}}$. The corresponding first-order formulation 
	\[
		{\dot {\bar y}} 
		= {\bar z}, \qquad
		{\bar z} + {\bar y} 
		= \shift{\tau}{{\bar z}} + \tau\shift{\tau}{{{\dot {\bar z}}}}
	\]
	equals a \DDE of advanced type, since the second equation involves
	$\shift{\tau}{\dot {\bar z}}$ but not $\dot {\bar z}$. 
	To illustrate the inherent difficulties of \DDEs of advances type, we set $\tau = 1$ and consider 
	\begin{displaymath}
	\phi(t) = y^{0} + \tfrac{1}{n}\sin(n\pi t) 
	\end{displaymath}
	for $t \in [-1, 0]$	as history function.
	The solution can be obtained using Bellman's method of
	steps~\cite[Ch.~3.4]{Bel61} and yields
	\begin{align*}
	{\bar y}(t) = 
	\begin{cases}
	y^{0} + \frac{1}{n} \sin(n\pi t), 
	& -1 \le t \le 0\\
	y^{0} \mathrm{e}^{-t} + \pi \cos(n \pi ) \mathrm{e}^{-t} -\pi\cos(n\pi(t-1)), 
	&\phantom{-}0 \le t \le 1\\
	y^{0} \mathrm{e}^{-t} + \pi\cos(n\pi)\mathrm{e}^{-t} - \pi \mathrm{e}^{-t+1} - n\pi^2\sin(n\pi(t-2)), 
	&\phantom{-}1 \le t \le 2\\
	y^{0} \mathrm{e}^{-t} + \pi\cos(n\pi)\mathrm{e}^{-t} - \pi \mathrm{e}^{-t+1} \\ 
	\qquad\qquad+ n^2\pi^2 \cos(n\pi)\mathrm{e}^{-t+2} - n^2\pi^3\cos(n\pi(t-3)), 
	& \phantom{-}2 \le t \le 3
	\end{cases}.
	\end{align*}
	Although the history function is uniformly bounded for all $n$, the solution in
	each successive interval grows with an increasing factor of $n$.
	In this sense, the delay system is ill-posed.
\end{example}
Solving \DDEs of advanced type requires a distributional solution concept even in the finite-dimensional setting~\cite{TreU19,Ung20b}. To overcome this difficulty, we instead aim for a higher-order approximation via multiple delays.

\subsection{Higher-order approximation via multiple delays}
We choose $\widehat{p}(t;\tau)$ in~\eqref{eq:ellpar:opt:delay} as the Lagrange interpolation polynomial of degree $\numDelay-1$ for the interpolation points $s_\ell\vcentcolon=t - \ell\tau$ given by
\begin{equation*}
	\widehat{p}(t;\tau) = \sum_{\ell=1}^{\numDelay}\, \bar{p}(s_\ell) L_\ell(t),
	\qquad\qquad 
	L_\ell(t) 
	= \prod_{j=1,\, j\neq \ell}^\delta \frac{t-s_j}{s_\ell - s_j}.
\end{equation*} 
Substituting the particular form of the interpolation points, we obtain
\begin{displaymath}
	L_\ell(t) \equiv (-1)^{(\ell-1)}\binom{\numDelay}{\ell}
	=\vcentcolon c_{\numDelay,\ell} \quad \ell=1,\dots,\numDelay.
\end{displaymath}
We would like to emphasize that $\bar{p}(s_\ell) = \bar{p}(t-\ell\tau)$ is not a constant but depends on $t$ and~$\tau$,
whereas the polynomial basis loses its dependence on $t$ and are constants, which we call \emph{delay coeffients};
see the examples in \Cref{tab:coeffDelays}.
Moreover, we have $\sum_{\ell=1}^\numDelay c_{\numDelay,\ell} = 1$ and obtain 
\begin{align}
\label{eq:p:approx:lagpol}
	{\widehat p}(t;\tau)\vcentcolon=\sum_{\ell=1}^{\numDelay}c_{\numDelay,\ell}\shift{\ell\tau}{\bar{p}}(t)
\end{align}
as the delay term. The proof that this is indeed an approximation of $p(t)$ of order $\delta$ is given in \Cref{lem:pTaylor}. 
\begin{table}[]
	\caption{Coefficients~$c_{\numDelay,\ell}$, $\ell=1,\dots,\numDelay$, for different values of $\numDelay$.}
	\label{tab:coeffDelays}
	\centering
	\begin{tabular}{l@{\quad}|@{\quad}c@{\quad}c@{\quad}c@{\quad}c}
		\toprule
		$\numDelay$ & $c_{\numDelay,1}$ & $c_{\numDelay,2}$ & $c_{\numDelay,3}$ & $c_{\numDelay,4}$\\	\midrule
		$1$ & $1$ & & &\\
		$2$ & $2$ & $-1$ & &\\
		$3$ & $3$ & $-3$ & $1$ &\\
		$4$ & $4$ & $-6$ & $4$ & $-1$\\
		\bottomrule
	\end{tabular}
\end{table}

\begin{example}
	Following \Cref{tab:coeffDelays}, the delay approximations for $\numDelay=1,2,3$ are given by
	\begin{align*}
		\widehat{p}(\,\cdot\,;\tau) &= \shift{\tau}{\bar{p}}, & 
		\widehat{p}(\,\cdot\,;\tau) &= 2\shift{\tau}{\bar{p}} - \shift{2\tau}{\bar{p}}, &
		\widehat{p}(\,\cdot\,;\tau) &= 3\shift{\tau}{\bar{p}} - 3\shift{2\tau}{\bar{p}} + \shift{3\tau}{\bar{p}},
	\end{align*}
	respectively. 
\end{example}
\begin{remark}
	The recipe in~\cite{AltMU24} to obtain the delay system is to replace the time derivatives of the delayed terms in the Taylor series expansion with finite differences of appropriate order such that the resulting approximation is still of order~$\numDelay$ but of neutral delay type. For $\numDelay\leq 3$, this approach resembles the schemes we obtain here, but for an approximation of order~$\tau^4$ already five delays are required. In contrast, we obtain a closed form expression for the delay term ${\hat p}(\,\cdot\,;\tau)$ with $\numDelay$ delays.
\end{remark}
Substituting the expression for $\widehat{p}(\,\cdot\,;\tau)$ from~\eqref{eq:p:approx:lagpol} in~\eqref{eq:ellpar:opt:delay} yields the system
\begin{subequations}
	\label{eq:ellpar:opt:delayK}
	\begin{align}
	\cA {\bar u} - \cD^{*} \bigg( \sum_{\ell=1}^{\numDelay} c_{\numDelay,\ell} \shift{\ell\tau}{\bar p} \bigg) 
	&= f \qquad \text{in } \cV^{*}, \label{eq:ellpar:opt:delayK:a}\\
	\cD{\dot {\bar u}} + \cC{\dot {\bar p}} + \cB {\bar p} 
	&= g \hspace{0.79cm} \text{in } \cQ^{*}\label{eq:ellpar:opt:delayK:b}
	\end{align}
\end{subequations}
with $\numDelay$ delays. 
For $\numDelay=1$, we recover the already mentioned delay system of order one introduced in \cite{AltMU21b}, and for $\numDelay\in\{2,3\}$ the delayed systems from \cite{AltMU24}.
In~\eqref{eq:ellpar:opt:delayK}, we seek solutions $\bar{u}\colon [0,T]\rightarrow\cV$ and $\bar{p}\colon [-\numDelay\tau,T]\rightarrow\cQ$ given sufficiently smooth right-hand sides, initial data, and a history function $\Phi \in C^{\infty}([-\numDelay\tau,0],\cHQ)$ for $\bar{p}$ that ensures the consistency of initial conditions of the original problem \eqref{eq:ellpar:opt}.
For this, a sufficient condition reads
\begin{align}
	\label{eq:histFunc}
	\Phi(-\ell\tau) = \Phi(0) = p^{0}\qquad \text{ for all } \ell=1, \cdots, \numDelay,
\end{align}
since this implies with~\eqref{eq:ellpar:opt:delayK:a} that $\cA {\bar u}(0) - \cD^{*} p^0 = \cA {\bar u}(0) - \Big( \sum_{\ell=1}^{\numDelay} c_{\numDelay,\ell}\Big) \cD^{*} p^0 = f(0)$ and hence, ${\bar u}(0) = u^0$. 

\begin{proposition}[Distance to delay solution]
\label{prop:distanceDelay}
Assume sufficiently smooth right-hand sides $f, g$, consistent initial data \eqref{eq:ellpar:c}, and a history function $\Phi \in C^{\infty}([-\numDelay\tau,0],\cHQ)$ which satisfies~\eqref{eq:histFunc}.
Then there exists a solution $({\bar u}, {\bar p})$ of the delay problem \eqref{eq:ellpar:opt:delayK} which satisfies ${\bar p} \in W^{\numDelay+1,\infty}(0,T;\cHQ)$. Moreover, the solutions of~\eqref{eq:ellpar:opt} and~\eqref{eq:ellpar:opt:delayK} only differ by a term of order~$\numDelay$, i.e., for almost every $t \in [0,T]$ it holds that
\begin{displaymath}
	\Vert {\bar u}(t) - u(t)\Vert_{\cV}^{2} + \Vert {\bar p}(t) - p(t)\Vert_{\cQ}^{2} \lesssim \tau^{2\numDelay}t \Big(\Vert {\bar p}\Vert_{W^{\numDelay+1,\infty}(0, t;\cHQ)}^{2} + \Vert {\Phi}\Vert_{W^{\numDelay+1, \infty}(-\numDelay\tau, 0;\cHQ)}^{2}\Big).
\end{displaymath}
\end{proposition}
\begin{proof}
	The boundedness of $\bar{p}$ is no different than the one considered in \cite{AltMU21b}, except for the specification of the history function. Hence, we refer to \cite[Prop.~A.\,4]{AltMU21b} for the proof.
	Now define $e_{p}\vcentcolon= {\bar p} - p$ and $e_{u}\vcentcolon={\bar u} - u$.
	Taking the difference of \eqref{eq:ellpar:opt} and \eqref{eq:ellpar:opt:delay} as well as the difference of the time derivatives of \eqref{eq:ellpar:opt:a} and \eqref{eq:ellpar:opt:delay:a}, \Cref{lem:pTaylor} yields
	\begin{subequations}
		\label{eq:proof:close:1}
		\begin{align}
			a(e_u, v) - d(v, e_p) &= \frac{1}{(\numDelay-1)!}\sum_{\ell=1}^{\numDelay}(-1)^{(\ell-1)}\binom{\numDelay}{\ell}\int_{t-\ell\tau}^{t}(t - \zeta)^{\numDelay-1} d\Big(v, {\bar p}^{(\numDelay)}(\zeta)\Big)\dzeta, \label{eq:proof:close:1:a}\\
			a({\dot e}_u, v) - d(v, {\dot e}_p) &= \frac{1}{(\numDelay-1)!}\sum_{\ell=1}^{\numDelay}(-1)^{(\ell-1)}\binom{\numDelay}{\ell}\int_{t-\ell\tau}^{t}(t - \zeta)^{\numDelay-1} d\Big(v, {\bar p}^{(\numDelay+1)}(\zeta)\Big)\dzeta, \label{eq:proof:close:1:c}
		\end{align}
		and 
		\begin{align}
			d({\dot e}_u, q) + c({\dot e}_p, q) + b(e_p, q) = 0 \label{eq:proof:close:1:b}
		\end{align}
	\end{subequations}
	for all test functions $v \in \cV$ and $q \in \cQ$.
	The sum of \eqref{eq:proof:close:1:b} and \eqref{eq:proof:close:1:c} tested with $v={\dot e}_u, q={\dot e}_p$ results in
	\begin{align*}
		\Vert {\dot e}_u\Vert_{a}^{2} + \Vert {\dot e}_p\Vert_{c}^{2} &+
		\tfrac{1}{2}\tddt \Vert {e}_p\Vert_{b}^{2} = \frac{1}{(\numDelay-1)!}\sum_{\ell=1}^{\numDelay}(-1)^{(\ell-1)}\binom{\numDelay}{\ell}\int_{t-\ell\tau}^{t}(t - \zeta)^{\numDelay-1} d\Big({\dot e}_u, {\bar p}^{(\numDelay+1)}(\zeta)\Big)\dzeta\nonumber\\
		&\le \frac{\tau^\numDelay}{(\numDelay-1)!}\sum_{\ell=1}^{\numDelay} \binom{\numDelay}{\ell} \ell^{\numDelay} C_d \Vert {\dot e}_u \Vert_{\cV} \Vert {\bar p}^{(\numDelay+1)}\Vert_{L^{\infty}(t-\ell\tau,t;\cHQ)}\nonumber\\
		&\le \frac{1}{2}\Vert {\dot e}_u \Vert_{\cV}^{2} + \frac{\numDelay}{2}\frac{\tau^{2\numDelay}}{((\numDelay-1)!)^{2}} C_d^2\sum_{\ell=1}^{\numDelay} \binom{\numDelay}{\ell}\binom{\numDelay}{\ell} \ell^{2\numDelay}\Vert {\bar p}^{(\numDelay+1)}\Vert_{L^{\infty}(t-\ell\tau,t;\cHQ)}^{2}.
	\end{align*}
	Integrating over $[0,t]$ and using the ellipticity properties of the bilinear forms yields
	\begin{align}
		\label{eq:proof:close:2}
		\int_{0}^{t} \Vert {\dot e}_u\Vert_{\cV}^{2}\ds + \int_{0}^{t} \Vert {\dot e}_p\Vert_{\cHQ}^{2}\ds + \Vert { e}_p\Vert_{\cQ}^{2} \lesssim \tau^{2\numDelay}t \Big(\Vert {\bar p}^{(\numDelay+1)}\Vert_{L^{\infty}(0, t;\cHQ)}^{2} + \Vert {\Phi}^{(\numDelay+1)}\Vert_{L^{\infty}(-\numDelay\tau, 0;\cHQ)}^{2}\Big).
	\end{align}
	Similarly the sum of \eqref{eq:proof:close:1:a} and \eqref{eq:proof:close:1:b} with the choice of test functions $v={\dot e}_u, q={e}_p$ results in
	\begin{multline*}
		\frac{1}{2}\tddt\Vert {e}_u\Vert_{a}^{2} + \frac{1}{2}\tddt\Vert {e}_p\Vert_{c}^{2} +
		\Vert {e}_p\Vert_{b}^{2} \\
		\le \frac{1}{2}\Vert {\dot e}_u \Vert_{\cV}^{2} + \frac{\numDelay}{2}\frac{\tau^{2\numDelay}}{((\numDelay-1)!)^{2}} C_d^2\sum_{\ell=1}^{\numDelay} \binom{\numDelay}{\ell}\binom{\numDelay}{\ell} \ell^{2\numDelay}\Vert {\bar p}^{(\numDelay)}\Vert_{L^{\infty}(t-\ell\tau,t;\cHQ)}^{2}.
	\end{multline*}
	Integrating once more over $[0,t]$ yields
	\begin{align}
		\label{eq:proof:close:3}
		\Vert { e}_u\Vert_{\cV}^{2} + \Vert { e}_p\Vert_{\cHQ}^{2} + \int_{0}^{t}\Vert { e}_p\Vert_{\cQ}^{2}\ds &\lesssim \int_{0}^{t} \Vert {\dot e}_u\Vert_{\cV}^{2}\ds + \tau^{2\numDelay}t\, \Vert {\bar p}^{(\numDelay)}\Vert_{L^{\infty}(0, t;\cHQ)}^{2}\nonumber\\
		&\qquad\qquad\qquad\quad + \tau^{2\numDelay}t\, \Vert {\Phi}^{(\numDelay)}\Vert_{L^{\infty}(-\numDelay\tau, 0;\cHQ)}^{2}.
	\end{align}
	From \eqref{eq:proof:close:2} and \eqref{eq:proof:close:3} we arrive at the assertion.
\end{proof}

Similar to~\eqref{eq:ppde}, we mention here the inherent delay parabolic equation
\begin{equation}
	\label{eqn:par:opt:delay}
	\cC\dot{\bar{p}} + \cM \bigg( \sum_{\ell=1}^{\numDelay} c_{\numDelay,\ell} 
					\shift{\ell\tau}{\dot{\bar{p}}} \bigg)  + \cB \bar{p} = r,
\end{equation}
which is obtained from~\eqref{eq:ellpar:opt:delayK} by solving for $\bar{u}$ in the first equation.
%
\section{Summation Assumption and Convergence Analysis} 
\label{sec:summation:bdf}

As discussed before, to obtain higher-order schemes for~\eqref{eq:ellpar:opt} which decouple, 
we apply an implicit discretization (of higher order) to the delay equation~\eqref{eq:ellpar:opt:delayK} 
with a constant time step size equal to the delay. More precisely, we discuss the application of the 
well-known \BDF schemes. For a sequence $(y^n)_{n\in\Z}$ and a constant step size $\tau$, we define 
the \BDF-$\numBDF$ operator 
\begin{equation}
	\label{eqn:BDFoperator}
	\bdf{\numBDF}{y^n} \vcentcolon= \frac{1}{\tau}\, 
			\sum_{\ell=0}^\numBDF \xi_{\numBDF,\ell}\, y^{n-\ell},
\end{equation}
where the coefficients $\xi_{\numBDF,\ell}$
are listed in \Cref{tab:coeffBDF} (cf.~\cite[Ch.~III.3]{HaiNW09}). 
Recall that \BDF methods up to order $\numBDF=6$ are stable~\cite[Ch.~III.3]{HaiNW09}. 
\begin{table}[]
	\renewcommand{\arraystretch}{1.25}
	\caption{\BDF-$\numBDF$ coefficients~$\xi_{\numBDF,\ell}$, $\ell=0,\dots,\numBDF$, for different values of $\numBDF$.}
	\label{tab:coeffBDF}
	\centering
	\begin{tabular}{@{\quad}c@{\quad}|rrrr}
		\toprule
		$k$ & \quad$\xi_{k,0}$ & \quad$\xi_{k,1}$ & \quad$\xi_{k,2}$ & \quad$\xi_{k,3}$\\	\midrule
		$1$ & $1$ & $-1$& &\\
		$2$ & $\tfrac{3}{2}$ & $-2$ & $\tfrac{1}{2}$&\\
		$3$ & $\tfrac{11}{6}$ & $-3$ & $\tfrac{3}{2}$& $-\tfrac{1}{3}$\\
		\bottomrule
	\end{tabular}
\end{table}
The decoupling time discretization scheme for~\eqref{eq:ellpar:opt}
is thus obtained by replacing the time derivatives in~\eqref{eq:ellpar:opt:delayK}
with the \BDF-$k$ operator, i.e., we consider the scheme
\begin{subequations}
	\label{eq:ellpar:opt:delay:BDF}
	\begin{align}
	\cA u^n - \cD^{*}  \delayOperKdisc{\numDelay}(p^{n})
	&= f^n \qquad \text{in } \cV^{*}, \label{eq:ellpar:opt:delayK:BDF:a}\\
	\cD \bdf{\numBDF}{u^n} + \cC \bdf{\numBDF}{p^n} + \cB p^n 
	&= g^n \hspace{0.79cm} \text{in } \cQ^{*}\label{eq:ellpar:opt:delayK:BDF:b}
	\end{align}
\end{subequations}
with the discrete delay operator
\begin{align}
	 \delayOperKdisc{\numDelay}(p^{n}) 
	 \vcentcolon= \sum_{\ell=1}^{\numDelay} c_{\numDelay,\ell}\, p^{n-\ell}.
	 \label{eq:delayOperator}
\end{align}
Note that~\eqref{eq:ellpar:opt:delay:BDF} defines a multistep scheme with $\max{(\numBDF, \numDelay)}$ steps.
\begin{remark}
Scheme~\eqref{eq:ellpar:opt:delay:BDF} can also be interpreted as an $(\alpha,\beta,\gamma)$-\BDF discretization of~\eqref{eq:ellpar:opt}. 
In \cite{Cro80,AkrCM98} the combination of \BDF-$k$ with an explicit multistep method of order $k$ is considered, based on the characteristic polynomials 
\[
\alpha(\zeta) = \sum_{j=1}^k \tfrac1j\, \zeta^{k-j}(\zeta-1)^j,\qquad
\beta(\zeta) = \zeta^k, \qquad
\gamma(\zeta) = \zeta^k - (\zeta-1)^k. 
\]
This means that the standard \BDF-$k$ method -- represented by $(\alpha,\beta)$ --  is applied to all but the extrapolated values. 
For the latter, which in our case corresponds to $p$ in~\eqref{eq:ellpar:opt:a}, a multistep discretization defined by $\gamma$ is used instead. 
Since the coefficients of $\gamma$ coincide with the delay coefficients given in \Cref{tab:coeffDelays}, this then yields~\eqref{eq:ellpar:opt:delay:BDF}.   
Unfortunately, known convergence results for $(\alpha,\beta,\gamma)$ schemes do not apply for the here-considered elliptic--parabolic systems. 
\end{remark}

\subsection{Equivalent formulation}
To simplify the convergence analysis, we make the following observation.
Solving~\eqref{eq:ellpar:opt:delayK:BDF:a} for $u^n$,
substituting the expression into~\eqref{eq:ellpar:opt:delayK:BDF:b},
and exploiting that the discrete delay operator and the \BDF-$\numBDF$
operator commute, we obtain
\begin{equation}
	\label{eqn:par:opt:delay:bdf}
	\cC\bdf{\numBDF}{p^n} + \cM \delayOperKdisc{\numDelay} \bdf{\numBDF}{p^{n}}
	+ \cB p^n = r^n
\end{equation}
with $\cM = \cD\cA^{-1}\cD^{*}$ as before and $r^n \vcentcolon=  g^n - \cD\cA^{-1}\dot{f}^n$.
We immediately notice that~\eqref{eqn:par:opt:delay:bdf} equals the \BDF-$k$ discretization with step size~$\tau$ of the operator delay equation~\eqref{eqn:par:opt:delay}.
In particular, we conclude that instead of showing that the
solution of~\eqref{eq:ellpar:opt:delay:BDF} converges to the 
solution of~\eqref{eq:ellpar:opt}, we can equivalently show 
that the solution of~\eqref{eqn:par:opt:delay:bdf} converges to the solution of~\eqref{eqn:par:opt:delay}. 
Note, however, that the combination of the delay and the \BDF approximation yields a $(\numDelay+\numBDF)$-step method for~$p$. Afterwards, $u$ can be reconstructed.
\begin{example}
	For the specific choice $\numDelay=\numBDF$ we obtain the schemes
	\begin{subequations}
		\begin{align*}
			\numBDF &= 1: & \cC p^{n} - (\cC - \cM) p^{n-1}
			- \cM p^{n-2} + \tau\cB p^{n} &= \tau r^{n},\\
			\notag\numBDF &= 2: & 3 \cC p^{n} - (4\cC - 6\cM) p^{n-1} 
			+ (\cC - 11\cM) p^{n-2} \qquad &\\
			&& +\,6\cM p^{n-3} -\cM p^{n-4} 
			+  2\tau\cB p^{n} &= 2\tau r^{n},\\
			\notag\numBDF &= 3: & 11 \cC p^{n} - (18\cC - 33\cM ) p^{n-1}
			+ (9\cC - 87\cM) p^{n-2}+ (-2\cC + 92\cM) p^{n-3}\qquad &\\
			\notag&&  -51 \cM p^{n-4} + 15\cM p^{n-5} -\,2\cM p^{n-6} +  6\tau\cB p^{n} &= 6\tau r^{n}.
		\end{align*}
	\end{subequations}
\end{example}

\subsection{Summation assumption}
For the ease of presentation, we introduce the following notation.
Given a real Hilbert space $\cX$, $s\in\N$, and a matrix~$G\in\spd{s}$,
where $\spd{s}$ is the set of real-valued symmetric positive definite matrices
of size $s\times s$, we define the weighted inner product
\begin{displaymath}
	\Vert Y \Vert_{G}^{2} \vcentcolon= 
			\langle G Y,Y\rangle_{\cX^{s}} \qquad\text{for } Y \in \cX^{s}.
\end{displaymath}
Moreover, for a sequence $(y^n)_{n\in\Z}$ with $y^n\in\cX$ and
numbers $\numBDF,\numDelay\in\N$, we define
\begin{align}
	\label{eqn:stackedVars}
	Y^{n} = \begin{bmatrix}
			y^{n}\\
			\vdots\\
			y^{n-\numBDF-\numDelay+1}
		\end{bmatrix} \in \cX^{\numBDF+\numDelay}.
\end{align}

For the convergence analysis, we need the following algebraic criterion that leads to $G$--stability. 
Here, we follow the multiplier technique of Nevanlinna and Odeh~\cite{NevO81} but formulate the criterion as an assumption which needs to be verified afterwards.
\begin{assumption}[Summation assumption]
	\label{ass:summation:bdf}
	Let $\cX$ be a Hilbert space and $\omega,\tau>0$.  
	Then there exist a parameter~$\eta\in[0,1)$ and for any $\mu\in[0,\omega]$
	a matrix $G\in\spd{\numBDF+\numDelay}$ 
	and a vector $\gamma = [\gamma_0,\ldots,\gamma_{\numBDF+\numDelay}]
	\in\R^{\numBDF+\numDelay+1}$ such that for any sequence
	$\{y^n\}_{n\in\Z}$ with $y^n\in\cX$ for $n\in\Z$ we have
	\begin{equation}
		\label{eq:summation:bdf}
		\begin{aligned}
			\Big\langle \tau\bdf{\numBDF}{y^{n}}
				+ \mu \delayOperKdisc{\numDelay}(\tau\bdf{\numBDF}{y^{n}})&,
				y^{n} - \eta y^{n-1}\Big\rangle_{\cX}
			= \Vert Y^{n} \Vert_{G}^{2} - \Vert Y^{n-1} \Vert_{G}^{2}
			+ \Big\| \sum_{i=0}^{\numBDF+\numDelay}\gamma_{i} y^{n-i} \Big\|_{\cX}^2,
		\end{aligned}
	\end{equation}
	where $Y^n$ is defined as in \eqref{eqn:stackedVars} and
	$\delayOperKdisc{\numDelay}$ is the discrete delay operator defined
	in~\eqref{eq:delayOperator}.
\end{assumption}
Let us emphasize that \Cref{ass:summation:bdf} only formally depends on $\tau$,
since $\tau\bdf{\numBDF}{\cdot}$ is independent of $\tau$,
and that the matrix $G$ as well as the vector $\gamma$ depend on $\mu$. 
In the forthcoming convergence proof, we need to estimate the matrix $G$ in terms of its smallest and largest eigenvalues and thus make the following assumption.
\begin{assumption}
	\label{ass:summation:bdf:niceEigvals}
	Let the notation be as in \Cref{ass:summation:bdf}. Then 
	\begin{align}
		c_G \vcentcolon= \inf_{\mu\in[0,\omega]} \lambda(G(\mu)) > 0\qquad\text{and}\qquad 
		C_G \vcentcolon= \sup_{\mu\in[0,\omega]} \lambda(G(\mu)) < \infty,
	\end{align}
	where $\lambda(G)$ is the set of eigenvalues of the symmetric matrix $G$.
\end{assumption}
%
%
\subsection{Convergence analysis}
\label{subsec:BDF:convergence}
In the following, we use \Cref{ass:summation:bdf,ass:summation:bdf:niceEigvals} to prove the convergence of the \BDF discretization for the delay system~\eqref{eqn:par:opt:delay}. 
The verification of these assumptions is then subject of the upcoming \Cref{subsec:assVeri}. 
\begin{theorem}
\label{thm:BDFdelay}
Consider the solution ${\bar p}$ of the delay equation~\eqref{eqn:par:opt:delay} 
for a sufficiently smooth right-hand side $r\colon[0,T]\rightarrow\cQ^*$.
Let $p^n$ be its approximation at time step $t^n$ 
given by~\eqref{eqn:par:opt:delay:bdf}.
If \Cref{ass:summation:bdf,ass:summation:bdf:niceEigvals} hold, then we get 
\begin{displaymath}
	\Vert p^{n} - {\bar p}(t^{n}) \Vert^{2}_{\cHQ} + 
		\tau\Vert p^{n} - {\bar p}(t^{n}) \Vert^{2}_{\cQ}
	\lesssim \tau^{2\numBDF}
	+ \sum_{\ell=0}^{k-1}\Vert p^{\ell} - {\bar p}(t^{\ell})\Vert_{\cHQ}^{2} 
	+ \tau\,\Vert p^{k-1} - {\bar p}(t^{k-1}) \Vert^{2}_{\cQ}.
\end{displaymath}
\end{theorem}

\begin{proof}
	We define the error and the defect by
	\begin{displaymath}
		e^{n} \vcentcolon= p^{n} - {\bar p}(t^{n}), \qquad
		d^{n}\vcentcolon= \frac{1}{\tau}\, \bigg(\sum_{\ell=0}^{\numBDF} \xi_{\numBDF,\ell} \bar{p}(t^{n-\ell})\bigg) - \dot{\bar{p}}(t^{n}).
	\end{displaymath} 
	Due to the construction of the \BDF operator, we immediately conclude that the defect satisfies the estimate
	$\Vert d^{n} \Vert_{\cHQ} \lesssim \tau^{\numBDF}$.
	Considering the difference between~\eqref{eqn:par:opt:delay} 
	and~\eqref{eqn:par:opt:delay:bdf} and testing with~$e^{n} - \eta e^{n-1}$, where $\eta$ is the multiplier from \Cref{ass:summation:bdf}, we obtain
	\begin{align}
			\Big\langle \cC\tau\bdf{\numBDF}{e^{n}}
				+ \cM \delayOperKdisc{\numDelay}(\tau\bdf{\numBDF}{e^{n}})
				&,e^{n}-\eta e^{n-1}\Big\rangle 
			+ \tau\, \langle \cB e^{n},e^{n}-\eta e^{n-1}\rangle \notag \\
			&\qquad\quad= \tau\, \Big\langle d^{n} 
				+ \cM \delayOperKdisc{\numDelay}({d^{n}})
				,e^{n} - \eta e^{n-1} \Big\rangle. \label{eqn:errDefectRelation}
	\end{align} 
	Since $\cB$ is self-adjoint, we have
	\begin{align*}
		2\, \langle \cB e^{n},e^{n}-\eta e^{n-1}\rangle
		= \| e^n\|_{b}^2 - \eta^2\, \| e^{n-1}\|_{b}^2 + \| e^n-\eta e^{n-1}\|_{b}^2
		\geq \| e^n\|_{b}^2 - \eta^2\,\| e^{n-1}\|_{b}^2.
	\end{align*}
	We now consider the spectral decomposition of the operator~$\cM$ w.r.t.~the inner product defined by the bilinear form~$c$. This means that we seek functions $\phi_i \in \cHQ$ satisfying 	
	\begin{displaymath}
		\cM \phi_{j} = \mu_j\cC\phi_{j}\qquad\text{and}\qquad \langle \cC \phi_j,\phi_i\rangle = \delta_{ji}.
	\end{displaymath}
	Following the spectral result in~\cite[Thm.~6.11]{Bre11}, this yields an \emph{orthonormal basis} of $\cHQ$ (w.r.t.~the $c$-norm). For this, we use that fact that $\cM$ is compact if interpreted as operator mapping from $\cQ$ to its dual. 
	Thus is due to the fact that this involves the compact embedding $\cV \hook \cHV$. For the eigenvalues, we conclude $0 \le  \mu_j \le \omega$, since the operator $\calC^{-1}\cM$ is bounded by $\omega$.  
	Now define $\epsilon_i^n \vcentcolon= \langle \cC e^n, \phi_i\rangle \in \R$
	and $E_i^n \vcentcolon= [\epsilon_i^n, \ldots, \epsilon_i^{n-\numBDF-\numDelay+1}]^\T \in \R^{\numBDF+\numDelay}$, leading to the basis expansion
	$e^{n} = \sum_{i=1}^{\infty}\epsilon^{n}_i \phi_{i}$. 
	Using that all eigenvalues are bounded
	above by $\omega$ and \Cref{ass:summation:bdf},
	the first term in~\eqref{eqn:errDefectRelation} can be rewritten as 
	\begin{align*}
		\Big\langle \cC\tau\bdf{\numBDF}{e^{n}} + 
		\cM \delayOperKdisc{\numDelay}(\tau\bdf{\numBDF}{e^{n}})
		, e^{n}-\eta e^{n-1}\Big\rangle_{\cHQ}
		&= \sum_{i=1}^{\infty}\Big\langle \tau\bdf{\numBDF}{\epsilon^{n}_i} + 
		 \mu_i \delayOperKdisc{\numDelay}(\tau\bdf{\numBDF}{\epsilon^{n}_i})
		 , \epsilon^{n}_i - \eta \epsilon^{n-1}_i\Big\rangle_{\R}\\
		 &\geq \sum_{i=1}^{\infty} \Big(\Vert E^{n}_{i} \Vert_{G(\mu_i)}^{2} 
		 - \Vert E^{n-1}_{i} \Vert_{G(\mu_i)}^{2}\Big).
	\end{align*}
	Due to the orthonormality of the $\phi_{i}$ functions,
	we have $\Vert e^{n}\Vert^{2}_{c} = \sum_{i=1}^{\infty}(\epsilon^{n}_i)^{2}$. 
	In the same way, it holds that
	\begin{align*}
		\Vert E^{n} \Vert^{2}_{c} = \Big\langle 
		\sum_{i=1}^{\infty}\cC E^{n}_i \phi_{i},
		\sum_{j=1}^{\infty}E^{n}_j \phi_{j}\Big\rangle_{\cHQ^{\numBDF+\numDelay}}
			= \sum_{i=1}^{\infty}\sum_{j=1}^{\infty}\langle \cC E^{n}_i \phi_{i}, 
			E^{n}_j \phi_{j}\rangle_{\cHQ^{\numBDF+\numDelay}} 
			= \sum_{i=1}^{\infty}\Vert E^{n}_{i} \Vert^{2}
	\end{align*}
	where $E_i^n\in\R^{\numBDF+\numDelay}$ and $E^{n}\in\cHQ^{\numBDF+\numDelay}$
	(see~\eqref{eqn:stackedVars}).
	Using the weighted Young's inequality for~\eqref{eqn:errDefectRelation}, we obtain with (yet to be determined constants)
	$\delta_1,\delta_2,\delta_3,\delta_4>0$, 
	\begin{align*}
		 \Big\langle d^{n} &+ \cM \Big(\sum_{\ell=1}^{\numDelay}
		 c_{\numDelay,\ell}d^{n-\ell}\Big), e^{n} - \eta e^{n-1} \Big\rangle\\
		 &= \langle d^n,e^n\rangle - \eta\, \langle d^n,e^{n-1}\rangle
		 + \sum_{\ell=1}^{\numDelay} c_{\numDelay,\ell} \Big(
			\langle \cM d^{n-\ell},e^{n}\rangle - \eta\,\langle \cM d^{n-\ell},
			e^{n-1}\rangle\Big)\\
		 &\leq \big(\tfrac{\delta_1}{2}\|e^n\|^2+ \tfrac{1}{2\delta_1}\|d^n\|^2\big)
		 + \eta\, \big( \tfrac{\delta_2}{2}\|e^{n-1}\|^2
		 	+ \tfrac{1}{2\delta_2}\|d^n\|^2\big)\\
		 &\qquad + \sum_{\ell=1}^{\numDelay} |c_{\numDelay,\ell}|\Big(
		 \tfrac{\delta_3}{2}\|e^n\|^2 + \tfrac{1}{2 \delta_3} \|\calM d^{n-\ell}\|^2
		 \Big)
		 + \eta\sum_{\ell=1}^{\numDelay} |c_{\numDelay,\ell}| \Big(
		\tfrac{\delta_4}{2}\|e^{n-1}\|^2 + \tfrac{1}{2 \delta_4} \|\calM d^{n-\ell}\|^2
		\Big)\\
		 &\leq \big(\tfrac{\delta_1}{2} + \tfrac{\bar{c}_{\numDelay}\delta_3}{2}\big)
		 \|e^{n}\|^2 + \eta\,\big(
			\tfrac{\delta_2}{2} + \tfrac{\bar{c}_{\numDelay}\delta_4}{2}
			\big) \|e^{n-1}\|^2 + \big(
				\tfrac{1}{2\delta_1} + \tfrac{\eta}{2\delta_2} \big)\|d^n\|^2\\
		 &\qquad + c_c^2\,\omega^2\big(\tfrac{1}{2\delta_3}
		 + \tfrac{\eta}{2\delta_4} \big)
		 \sum_{\ell=1}^{\numDelay} |c_{\numDelay,\ell}| \|d^{n-\ell}\|^2,
	\end{align*}
	where $\bar{c}_{\numDelay} \vcentcolon= 
	\sum_{\ell=1}^{\numDelay} |c_{\numDelay,\ell}|$.
	Note that we have made use of
	$\Vert\cM \cdot\Vert \le c_c\,\omega \Vert\cdot\Vert$ 
	in the last inequality.
	With the choice
	\begin{align*}
		\delta_1 &= \frac{c_b}{2}(1 - \eta^2), &
		\delta_2 &= \frac{\eta}{2}, &
		\delta_3 &= \frac{c_b}{2}\frac{1}{\bar{c}_{\numDelay}}(1 - \eta^2), &
		\delta_4 &= \frac{\eta}{2}\frac{1}{\bar{c}_{\numDelay}},
	\end{align*}
 	we obtain the estimate 
	\begin{align*}
		\sum_{i=1}^{\infty} \Bigg(\Vert E^{n}_{i} \Vert_{G(\mu_i)}^{2}
		- \Vert E^{n-1}_{i} \Vert_{G(\mu_i)}^{2}&\Bigg)
		+ \tfrac{1}{2}\tau\eta^2\| e^n\|_{b}^2 - \tfrac{1}{2}\tau\eta^2\| e^{n-1}\|_{b}^2 - \tfrac{1}{2}\tau\eta^2\| e^{n-1}\|^2\\
		&\le \Big(1 + \tfrac{1}{c_b (1 - \eta^2)}\Big)
		\Bigg(
			\tau\, \Vert d^{n} \Vert^{2} +
		\tau c_c^2\,\omega^2\bar{c}_{\numDelay}
		\sum_{\ell=1}^{\numDelay}\vert c_{\numDelay,\ell}\vert \Vert d^{n-\ell} \Vert^{2} 
		\Bigg)\nonumber\\
		&\lesssim (1 + \omega^2)\, \tau^{2\numBDF + 1}.
	\end{align*}
	Using \Cref{ass:summation:bdf:niceEigvals}, we observe that
	\begin{align*}
		\Vert e^{n-1} \Vert^{2}
		\le \tfrac{1}{c_c}\Vert e^{n-1} \Vert^{2}_{c}
		= \tfrac{1}{c_c}\sum_{i=1}^{\infty}(\epsilon^{n-1}_i)^{2}
		\le \tfrac{1}{c_c}\sum_{i=1}^{\infty}\tfrac{1}{c_{G(\mu_i)}}
		\Vert E^{n-1}_{i} \Vert_{G(\mu_i)}^{2}
		\le \tfrac{1}{c_c c_G}\sum_{i=1}^{\infty}
		\Vert E^{n-1}_{i} \Vert_{G(\mu_i)}^{2}.
	\end{align*}
	Overall, this yields the perturbed telescope sum
	\begin{align*}
		\sum_{i=1}^{\infty} \Bigg(\Vert E^{n}_{i} \Vert_{G(\mu_i)}^{2}
		- (1 + \tau C)\Vert E^{n-1}_{i} \Vert_{G(\mu_i)}^{2}\Bigg)
		 + \tfrac{1}{2}\tau\eta^2 \Big(\| e^n\|_{b}^2 - \| e^{n-1}\|_{b}^2\Big)
		\lesssim (1 + \omega^2)\, \tau^{2\numBDF + 1},
	\end{align*}
	where $C = \eta^2\tfrac{1}{2 c_c c_G}$.
	To see that the Gr\"{o}nwall-type inequality from \cite[Lem.~2.5]{AltMU24b}
	can be applied, define ${a}^{n} \vcentcolon= \sum_{i=1}^{\infty} \Vert E^{n}_{i} \Vert_{G(\mu_i)}^{2}$. 
	This then leads to 
	\begin{align*}
		\sum_{i=1}^{\infty} \Vert E^{n}_{i} \Vert_{G(\mu_i)}^{2}
		 + \tau\, \Vert e^{n} \Vert^{2}_{b}
		 \lesssim e^{C n \tau} 
		 \Bigg(\sum_{i=1}^{\infty} \Vert E^{\numBDF-1}_{i} \Vert_{G(\mu_i)}^{2} + \tau\, \Vert e^{k-1} \Vert^{2}_{b}
		 + t^{n}\tau^{2\numBDF}
		 \Bigg).
	\end{align*}
	Together with
	\begin{align*}
		\sum_{i=1}^{\infty} \Vert E^{n}_{i} \Vert_{G(\mu_i)}^{2}
		\ge c_G \sum_{i=1}^{\infty} \Vert E^{n}_{i} \Vert^{2}
		= c_G \Vert E^{n}\Vert^{2}_{c}
		\ge c_G \Vert e^{n}\Vert^{2}_{c}
	\end{align*}
	and
	\begin{align*}
		\sum_{i=1}^{\infty} \Vert E^{\numBDF-1}_{i} \Vert_{G(\mu_i)}^{2}
		\le C_G \sum_{i=1}^{\infty} \Vert E^{\numBDF-1}_{i} \Vert^{2}
		= C_G\, \Vert E^{\numBDF-1}\Vert^{2}_{c}
		= C_G\, \Big(\Vert e^{\numBDF-1}\Vert^{2}_{c} + \cdots
					+ \Vert e^{-\numDelay-1}\Vert^{2}_{c} \Big)
	\end{align*}
	we obtain the desired estimate using the history function given in~\eqref{eq:histFunc}.
\end{proof}
\begin{remark}
	\label{rem:errorBalance}
	Suppose that the assumptions from \Cref{prop:distanceDelay}
	and~\Cref{ass:summation:bdf,ass:summation:bdf:niceEigvals} hold.
	Let $p$ be the solution of the original system~\eqref{eq:ppde} and $p^n$ the
	numerical approximation at time $t^n$ given by~\eqref{eqn:par:opt:delay:bdf}.
	Then the difference satisfies
	\begin{displaymath}
		\Vert u^{n} - {u}(t^{n}) \Vert^{2}_{\cV} + 
		\Vert p^{n} - {p}(t^{n}) \Vert^{2}_{\cHQ}
		\lesssim \tau^{2\numDelay} + \tau^{2\numBDF}
		+ \sum_{\ell=0}^{k-1}\Vert p^{\ell} - {p}(t^{\ell})\Vert_{\cHQ}^{2},
	\end{displaymath}
	indicating that we obtain an optimal (balanced) error for
	the choice $\numDelay = \numBDF$.
\end{remark}

\subsection{Verification of the summation assumption}
\label{subsec:assVeri}

It remains to validate \Cref{ass:summation:bdf,ass:summation:bdf:niceEigvals}. 
Since our main goal is the error analysis for the decoupling integration scheme, we follow \Cref{rem:errorBalance} and consider $\numBDF = \numDelay$. 
%
%
\subsubsection{\texorpdfstring{Case $\numBDF=\numDelay=1$}{Case k=d=1}}
Rewriting \eqref{eq:summation:bdf} for $\numBDF=\numDelay=1$, we get 
\begin{align*}
	\langle y^{n} - &y^{n-1} + \mu ({y^{n-1}} - y^{n-2}), y^{n} - \eta y^{n-1}\rangle\\
	&\hspace{3em}= \left\langle \begin{bmatrix}g_{11} & g_{12}\\g_{12} & g_{22}\end{bmatrix}\begin{bmatrix}
	y^{n}\\
	y^{n-1}
\end{bmatrix},\begin{bmatrix}
	y^{n}\\
	y^{n-1}
\end{bmatrix}\right\rangle - \left\langle \begin{bmatrix}g_{11} & g_{12}\\g_{12} & g_{22}\end{bmatrix}\begin{bmatrix}
	y^{n-1}\\
	y^{n-2}
\end{bmatrix},\begin{bmatrix}
	y^{n-1}\\
	y^{n-2}
\end{bmatrix}\right\rangle \\
&\qquad\qquad\qquad+ \big\| \gamma_0 y^{n} + \gamma_1 y^{n-1} + \gamma_{2} y^{n-2}\big\|^2,
\end{align*}
which on rearranging gives
\begin{align*}
	&(g_{11} + \gamma_{0}^{2} - 1)\langle y^{n}, y^{n}\rangle +(\eta \mu - \eta - g_{11} + g_{22} + \gamma_{1}^{2})\langle y^{n-1}, y^{n-1}\rangle\nonumber\\
	&\quad + (- g_{22} + \gamma_{2}^{2})\langle y^{n-2}, y^{n-2}\rangle +(\eta + 2 g_{12} + 2 \gamma_{0} \gamma_{1} - \mu + 1)\langle y^{n}, y^{n-1}\rangle \nonumber\\
	&\qquad + (2 \gamma_{0} \gamma_{2} + \mu)\langle y^{n}, y^{n-2}\rangle + (- \eta \mu - 2 g_{12} + 2 \gamma_{1} \gamma_{2})\langle y^{n-1}, y^{n-2}\rangle
	= 0.
\end{align*}
Equating the coefficients of each distinct inner product to zero, we obtain the system of equations
\begin{subequations}
	\label{eqn:k1d1:full}
	\begin{align}
		g_{11} + \gamma_{0}^{2} &= 1, &
		- g_{11} + g_{22} + \gamma_{1}^{2} &= \eta(1 - \mu), &
		2 g_{12} + 2 \gamma_{0} \gamma_{1} &=  -1 + \mu -\eta,\label{eqn:k1d1:full:a}\\
		- g_{22} + \gamma_{2}^{2} &= 0, &
		2 \gamma_{0} \gamma_{2} &= -\mu, &
		- 2 g_{12} + 2 \gamma_{1} \gamma_{2} &= \eta \mu. \label{eqn:k1d1:full:b}
	\end{align}
\end{subequations}
Solving for the variables $g_{11}, g_{22}, g_{12}$ in terms of $\gamma_{0}, \gamma_{1}, \gamma_{2}$
from \eqref{eqn:k1d1:full:a}, we obtain
\begin{displaymath}
	G = \begin{bmatrix}
	1 - \gamma_{0}^{2} &  - \gamma_{0} \gamma_{1} + \tfrac{1}{2}(\mu -\eta -1)\\
	- \gamma_{0} \gamma_{1} + \tfrac{1}{2}(\mu -\eta -1) &
	- \gamma_{0}^{2} - \gamma_{1}^{2} + 1 - \eta \mu + \eta
	\end{bmatrix}
\end{displaymath}
and the $\gamma_i$'s have to be solved from
\begin{subequations}
	\label{eqn:k1d1}
	\begin{align}
		2 \gamma_{0} \gamma_{2} &= - \mu, \label{eqn:k1d1:a}\\
		2 \gamma_{0} \gamma_{1} + 2 \gamma_{1} \gamma_{2}
			&= \eta \mu - \eta + \mu - 1, \label{eqn:k1d1:b}\\
		\gamma_{0}^{2} + \gamma_{1}^{2} + \gamma_{2}^{2}
			&= 1 - \eta \mu + \eta.\label{eqn:k1d1:c}
	\end{align}
\end{subequations}
Summing up the equations in~\eqref{eqn:k1d1} yields
$(\sum_{i=0}^3 \gamma_i)^2 = 0$ and we conclude that 
\begin{equation}
	\label{eqn:k1d1:sumgamma}
	\gamma_0 + \gamma_1 + \gamma_2 = 0.
\end{equation}
From \eqref{eqn:k1d1:b} and~\eqref{eqn:k1d1:sumgamma} we further obtain
\begin{align*}
	\gamma_1^2 = \tfrac{1}{2}(1+\eta)(1-\mu).
\end{align*}
Hence, there cannot exist a (real-valued) solution for $\mu>1$, i.e., we recover the weak coupling condition from \cite{AltMU21b}.
From~\eqref{eqn:k1d1:a} and~\eqref{eqn:k1d1:sumgamma} we solve for $\gamma_0$ and $\gamma_2$ to obtain
\begin{align*}
	\gamma_0 &= - \tfrac{1}{2}\sqrt{\tfrac{1}{2}(1+\eta)(1 - \mu)} + \sqrt {\tfrac{1}{2}\mu + \tfrac{1}{8}(1+\eta)(1 - \mu)}, \\
	\gamma_2 &= - \tfrac{1}{2}\sqrt{\tfrac{1}{2}(1+\eta)(1 - \mu)} - \sqrt {\tfrac{1}{2}\mu + \tfrac{1}{8}(1+\eta)(1 - \mu)},\\
	\gamma_1 &= \sqrt{\tfrac{1}{2}(1+\eta)(1 - \mu)}
\end{align*}
as a solution. Setting $\eta=0$ in the above expressions, the eigenvalues of the $G$ matrix are
\begin{align*}
	\lambda_1 &= \tfrac{1}{4}\sqrt{1 - \mu} \sqrt{3 \mu + 1} - \tfrac{1}{4}\sqrt{(3 + \mu)(1-\mu) + 2 (1 - \mu) \sqrt{1 - \mu} \sqrt{3 \mu + 1}} + \tfrac{1}{2},\\
	\lambda_2 &= \tfrac{1}{4}\sqrt{1 - \mu} \sqrt{3 \mu + 1} + \tfrac{1}{4}\sqrt{(3 + \mu)(1-\mu) + 2 (1 - \mu) \sqrt{1 - \mu} \sqrt{3 \mu + 1}} + \tfrac{1}{2}
\end{align*}
and a necessary and sufficient condition for $\lambda_1, \lambda_2 \in \R_{+}$ is $0\leq \mu\leq 1$. We conclude that \Cref{ass:summation:bdf,ass:summation:bdf:niceEigvals} are satisfied.
We refer to \Cref{fig:bdfeigvals:k1} for a visualization of the eigenvalues. 
%
%
\subsubsection{\texorpdfstring{Case $\numBDF=\numDelay=2$}{Case k=d=2}}
Equating the coefficients of all distinct inner products in \eqref{eq:summation:bdf} results in a system with $15$ unknowns and $15$ equations (while treating $\mu$ and $\eta$ as free parameters). Similarly as in the previous subsection, we can write all entries of $G$ as functions of the $\gamma_i$. The details are given in \Cref{subsec:k2d2:G}. 
In turn, $\gamma_0,\ldots,\gamma_4$ have to satisfy the following five nonlinear equations
\begin{subequations}
	\label{eqn:k2d2}
	\begin{align}
		\label{eqn:k2d2:a}2\gamma_0\gamma_4 &= -\tfrac{1}{2}\mu,\\
		\label{eqn:k2d2:b}2\gamma_0\gamma_3 + 2\gamma_1\gamma_4 
		&= \tfrac{1}{2}\eta \mu + 3\mu,\\
		\label{eqn:k2d2:c}2\gamma_0\gamma_2 + 2\gamma_1\gamma_3 + 2\gamma_2\gamma_4
		&= - 3 \eta \mu -\tfrac{11}{2}\mu + \tfrac{1}{2},\\
		\label{eqn:k2d2:d}2\gamma_0\gamma_1 + 2\gamma_1\gamma_2
		+ 2\gamma_2\gamma_3 + 2\gamma_3\gamma_4
		&= \tfrac{11}{2}\eta \mu - 2 \eta + 3\mu - 2,\\
		\label{eqn:k2d2:e}\gamma_0^2 + \gamma_1^2 + \gamma_2^2
		+ \gamma_3^2 + \gamma_4^2 &= -3 \eta \mu + 2 \eta + \tfrac{3}{2}.
	\end{align}
\end{subequations}

\begin{lemma}
	Consider~\eqref{eqn:k2d2} and assume $\eta\geq 0$.
	Then,~\eqref{eqn:k2d2} is not solvable for $\mu > \tfrac{1}{3}$.
	For $\mu \leq \tfrac{1}{3}$ a solution is given by
	\begin{subequations}
		\label{eqn:k2d2:gammaSol}
	\begin{align}
		\gamma_0 &= \tfrac{\theta}{2} + \tfrac{\sqrt{\mu + \theta^{2}}}{2}, &
		\gamma_2 &= - \tfrac{3 \sqrt{1 - 3 \mu}}{4} + \tfrac{\sqrt{5 \mu + 1}}{4}, &
		\gamma_4 &= \tfrac{\theta}{2} - \tfrac{\sqrt{\mu + \theta^{2}}}{2}, \\
		\gamma_1 &= \tfrac{\sqrt{1 - 3 \mu}}{2} - \tfrac{3 \mu - \sqrt{1 - 3 \mu} \theta}{2 \sqrt{\mu + \theta^{2}}}, &
		\gamma_3 &= \tfrac{\sqrt{1 - 3 \mu}}{2} + \tfrac{3 \mu - \sqrt{1 - 3 \mu} \theta}{2 \sqrt{\mu + \theta^{2}}} &
		\eta &= 0
	\end{align}
	\end{subequations}
	with $\theta = -\tfrac{1}{4} \big( \sqrt{5\mu+1} + \sqrt{1-3\mu} \big)$.
	For this particular solution, the matrix $G$ is positive definite for every $\mu\in[0,\tfrac{1}{3}]$.
	Hence, \Cref{ass:summation:bdf,ass:summation:bdf:niceEigvals} are true for $\numBDF = \numDelay = 2$
	and~$\omega\leq \tfrac{1}{3}$.
\end{lemma}
\begin{proof}
	Summing up the equations in~\eqref{eqn:k2d2} yields
	$(\sum_{i=0}^4 \gamma_i)^2 = 0$ and, hence, $\sum_{i=0}^4 \gamma_i = 0$.
	Factoring the sum of equations \eqref{eqn:k2d2:b} and \eqref{eqn:k2d2:d} gives
	\begin{align*}
		2(\gamma_1 + \gamma_3)(\gamma_0 + \gamma_2 + \gamma_4)
		&= 6\eta\mu + 6\mu -2\eta - 2 = 2(1+\eta)(3\mu - 1).
	\end{align*}
	Together with the fact that the $\gamma_i$'s sum up to zero, we obtain
	\begin{align*}
		(\gamma_1+\gamma_3)^2 = (1+\eta)(1-3\mu).
	\end{align*}
	We conclude that a (real-valued) solution cannot exist for
	$\mu>\tfrac{1}{3}$, which establishes the first claim.
	The second claim is obtained by directly verifying~\eqref{eqn:k2d2}.
	In more detail, we immediately see that~\eqref{eqn:k2d2:gammaSol}
	satisfies~\eqref{eqn:k2d2:a}. For~\eqref{eqn:k2d2:b}, we compute
	\begin{align*}
		2\gamma_0\gamma_3 + 2\gamma_1\gamma_4
		&= 4\tfrac{\theta}{2}\tfrac{\sqrt{1-3\mu}}{2}
		+ 4\tfrac{\sqrt{\mu + \theta^{2}}}{2}\tfrac{3\mu -
		\theta\sqrt{1-3\mu}}{2 \sqrt{\mu + \theta^{2}}} = 3\mu.
	\end{align*}
	Observe that
	\begin{align*}
		\frac{(3\mu - \theta\sqrt{1 - 3 \mu})^2}{\mu + \theta^2}
		 &= \frac{
			\tfrac{1}{8}(1 + 10\mu + 33\mu^2) + \tfrac{1}{8}(1+9\mu)\sqrt{1+5\mu}\sqrt{1-3\mu}
			}{
				\tfrac{1}{8}(1+9\mu) + \tfrac{1}{8}\sqrt{1+5\mu}\sqrt{1-3\mu}
			}\\
		 &= \frac{
			\mu(1+6\mu)(1+9\mu) + \mu(1+6\mu)\sqrt{1+5\mu}\sqrt{1-3\mu}
		 }{
			2\mu(1+6\mu)
		 }\\
		 &= \tfrac{1}{2}(1+9\mu) + \tfrac{1}{2}\sqrt{1+5\mu}\sqrt{1-3\mu}.
	\end{align*}
	Towards~\eqref{eqn:k2d2:c}, we obtain
	\begin{align*}
		(\gamma_0 + \gamma_4)\gamma_2 
		&= -\tfrac{3\theta \sqrt{1 - 3 \mu}}{4} 
		+\tfrac{\theta\sqrt{5 \mu + 1}}{4}
		= \tfrac{1 - 7 \mu}{8}
		+ \tfrac{\sqrt{1 - 3 \mu} \sqrt{5 \mu + 1}}{8},\\
		\gamma_1\gamma_3 &= \tfrac{1-3\mu}{4} - \tfrac{(3\mu - \theta\sqrt{1 - 3 \mu})^2}{4(\mu+\theta^2)}
		= -\tfrac{\sqrt{1 - 3 \mu} \sqrt{5 \mu + 1}}{8} +\tfrac{1-15\mu}{8},\\
		\gamma_0\gamma_2 + \gamma_1\gamma_3 + \gamma_2\gamma_4
		&= \tfrac{1-7\mu}{8} + \tfrac{1-15\mu}{8}= \tfrac{1}{4} - \tfrac{11}{4}\mu.
	\end{align*}
	Towards \eqref{eqn:k2d2:d}, we have
	\begin{align*}
		\gamma_2(\gamma_1 + \gamma_3) + \gamma_0\gamma_1 + \gamma_3\gamma_4
		&= (- \tfrac{3 \sqrt{1 - 3 \mu}}{4} + \tfrac{\sqrt{5 \mu + 1}}{4})\sqrt{1-3\mu}
		+ \tfrac{\theta}{2}\sqrt{1 - 3 \mu} - \tfrac{3\mu - \theta\sqrt{1 - 3 \mu}}{2}\\
		&= (- \sqrt{1 - 3 \mu} - \theta)\sqrt{1-3\mu} + \tfrac{\theta}{2}\sqrt{1 - 3 \mu}
		- \tfrac{3\mu - \theta\sqrt{1 - 3 \mu}}{2}\\
		&= -1 + \tfrac{3}{2}\mu.
	\end{align*}
	Finally, for \eqref{eqn:k2d2:e}, we have
	\begin{align*}
		\sum_{i=0}^{4}\gamma_i^{2} &= 2(\tfrac{\theta^2}{4} + \tfrac{\mu + \theta^{2}}{4})
		+ 2(\tfrac{1 - 3 \mu}{4} + \tfrac{(3\mu - \theta^2\sqrt{1 - 3 \mu})^2}{4(\mu + \theta^{2})})
		+ (- \tfrac{3 \sqrt{1 - 3 \mu}}{4} + \tfrac{\sqrt{5 \mu + 1}}{4})^2\\
		&= \tfrac{\mu}{2} + \tfrac{1}{16}(1 - 3 \mu + 1 + 5\mu + 2\sqrt{1 - 3 \mu} \sqrt{5 \mu + 1}) \\
		&\qquad + \tfrac{1 - 3 \mu}{2}
		+ \tfrac{1}{4}(1+9\mu) + \tfrac{1}{4}\sqrt{1+5\mu}\sqrt{1-3\mu} \\
		&\qquad\quad+ \tfrac{1}{16}(9 - 27 \mu + 1 + 5\mu - 6\sqrt{1 - 3 \mu} \sqrt{5 \mu + 1})\\
		&= \tfrac{1}{2} - \mu + \tfrac{1}{16}(12 - 20\mu) + \tfrac{1}{4}(1+9\mu) = \tfrac{3}{2}.
	\end{align*}
For the positive definiteness of the $G$ matrix,
we substitute the above analytical expressions for the $\gamma_i$'s
and numerically compute the eigenvalues of the parametrized $G$ matrix w.r.t.~$\mu$.
One can see in~\Cref{fig:bdfeigvals:k2} that the eigenvalues are all bounded from above and away from zero.
\end{proof}
\begin{figure}[ht]
	\centering
	\ref*{legEigvals}\\
	\begin{subfigure}[t]{.33\linewidth}
		\centering
		\begin{tikzpicture}
    \begin{axis}[
    width=2in,
    height=2in,
    xlabel={$\mu$},
    ymin=0, ymax=1.5,
    xmin=0,
    grid=major,
    cycle list name=summationsol,
    ]
    \addplot+[domain=0:1, samples=101] {(1/4)*sqrt(1 - x) * sqrt(3 *x + 1) - (1/4) * sqrt((3 + x)*(1-x) + 2 * (1 - x) * sqrt(1 - x) * sqrt(3 * x + 1)) + (1/2)};
    \addplot+[domain=0:1, samples=101] {(1/4)*sqrt(1 - x) * sqrt(3 *x + 1) + (1/4) * sqrt((3 + x)*(1-x) + 2 * (1 - x) * sqrt(1 - x) * sqrt(3 * x + 1)) + (1/2)};
    \addplot+[dashed, thick, black] coordinates {(1.0, 0.0) (1.0, 1.5)};
\end{axis}
\end{tikzpicture}
		\caption{$\numBDF=1$}
		\label{fig:bdfeigvals:k1}
	\end{subfigure}\hfill
	\begin{subfigure}[t]{.33\linewidth}
		\centering
		\begin{tikzpicture}
    \begin{axis}[
    width=2in,
    height=2in,
    xlabel={$\mu$},
    xmin = 0,
    ymin=1e-3, ymax=8,
    ymode={log},
    grid=major,
    cycle list name=summationsol,
    ]
    \addplot+ table[x=omega, y=lambda_1] {df_bdf2.dat};
    \addplot+ table[x=omega, y=lambda_2] {df_bdf2.dat};
    \addplot+ table[x=omega, y=lambda_3] {df_bdf2.dat};
    \addplot+ table[x=omega, y=lambda_4] {df_bdf2.dat};
    \addplot+[dashed, thick, black] coordinates {(1.0/3.0, 1.0e-3) (1.0/3.0, 10)};
\end{axis}
\end{tikzpicture}
		\caption{$\numBDF=2$}
		\label{fig:bdfeigvals:k2}
	\end{subfigure}\hfill
	\begin{subfigure}[t]{.33\linewidth}
		\centering
		\begin{tikzpicture}
    \begin{axis}[
    width=2in,
    height=2in,
    xlabel={$\mu$},
    xmin=0, xmax=0.16,
    ymin=1e-3, ymax=10,
    xtick={0.0, 0.05, 0.1, 0.15, 0.2},
    xticklabels={$0$, $0.05$, $0.1$, $0.15$, $0.2$},
    ymode={log},
    grid=major,
    legend entries={$\lambda_1$, $\lambda_2$, $\lambda_3$, $\lambda_4$, $\lambda_5$, $\lambda_6$}, 
    cycle list name=summationsol,
    legend cell align={left},
    legend style={
      /tikz/every even column/.append style={column sep=0.5cm},
      font=\small,
    },
    legend columns=6,
    legend to name=legEigvals, 
    ]
    \addplot+ table[x=omega, y=lambda_1] {df_bdf3.dat};
    \addplot+ table[x=omega, y=lambda_2] {df_bdf3.dat};
    \addplot+ table[x=omega, y=lambda_3] {df_bdf3.dat};
    \addplot+ table[x=omega, y=lambda_4] {df_bdf3.dat};
    \addplot+ table[x=omega, y=lambda_5] {df_bdf3.dat};
    \addplot+ table[x=omega, y=lambda_6] {df_bdf3.dat};
    \addplot+[dashed, thick, black] coordinates {(1.0/7.0, 1e-15) (1.0/7.0, 100)};
\end{axis}
\end{tikzpicture}
		\caption{$\numBDF=3$}
		\label{fig:bdfeigvals:k3}
	\end{subfigure}
	\caption{Eigenvalues of $G(\mu)$ for different $\numBDF$. The dashed lines indicate the critical values from \Cref{th:mainResult}, i.e., $\omega=1$, $\omega=\tfrac{1}{3}$, and $\omega=\tfrac{1}{7}$.}
	\label{fig:bdfeigvals}
\end{figure}
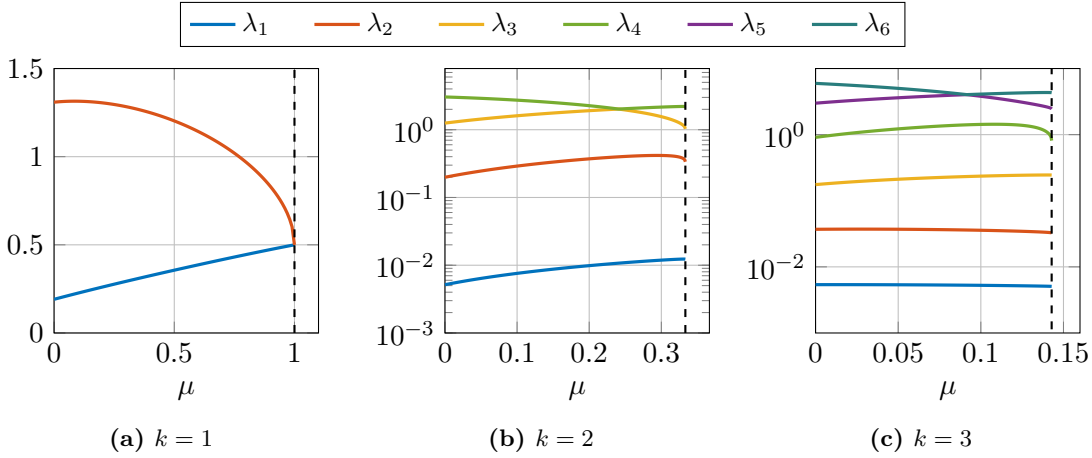

\begin{remark}
	The condition $\mu\leq \tfrac{1}{3}$ improves upon the existing bound $\mu\leq \tfrac{1}{5}$ used in the convergence proof in \cite{AltMU24} and is in agreement with the asymptotic stability analysis for the delay equation. We refer to  \cite[Sec.~3.5]{AltMU24} for further details.
\end{remark}

\begin{remark}
	There exists other solutions of~\eqref{eqn:k2d2} such that the eigenvalues of the $G$ matrix are not bounded away from zero, namely 
	\begin{align*}
		\begin{aligned}
		\gamma_0 &= \tfrac{\theta}{2} + \tfrac{\sqrt{\mu + \theta^{2}}}{2}, &
		\gamma_2 &= - \tfrac{3 \sqrt{1 - 3 \mu}}{4} - \tfrac{\sqrt{5 \mu + 1}}{4}, &
		\gamma_4 &= \tfrac{\theta}{2} - \tfrac{\sqrt{\mu + \theta^{2}}}{2}, \\
		\gamma_1 &= \tfrac{\sqrt{1 - 3 \mu}}{2} + \tfrac{\theta\sqrt{1 - 3 \mu} - 3\mu}{2 \sqrt{\mu + \theta^{2}}}, &
		\gamma_3 &= \tfrac{\sqrt{1 - 3 \mu}}{2} - \tfrac{\theta\sqrt{1 - 3 \mu} - 3\mu}{2 \sqrt{\mu + \theta^{2}}}, &
		\eta &= 0
		\end{aligned}
	\end{align*}
	with $\theta=\tfrac{1}{4} \big( \sqrt{5\mu+1} - \sqrt{1-3\mu} \big)$.
\end{remark}
%
\subsubsection{\texorpdfstring{Case $\numBDF=\numDelay=3$}{Case k=d=3}}
Equating the coefficients of all distinct inner products in \eqref{eq:summation:bdf} results in a system of $28$ unknowns and $28$ equations, again treating $\mu$ and $\eta$ as free parameters.
For~\eqref{eq:summation:bdf} to hold, we obtain $G$ as in~\eqref{eqn:k3d3:G} such that $\gamma_i$, $0\le i\le 6$, have to satisfy the following (nonlinear) system of equations
\begin{subequations}
	\label{eqn:k3}
	\begin{align}
		\label{eqn:k3:a} 2 \gamma_{0} \gamma_{6} &= - \tfrac{1}{3}\mu,\\
		\label{eqn:k3:b} 2 \gamma_{0} \gamma_{5} + 2 \gamma_{1} \gamma_{6}
		&= \tfrac{1}{3}\eta \mu + \tfrac{5}{2}\mu,\\
		\label{eqn:k3:c} 2 \gamma_{0} \gamma_{4} + 2 \gamma_{1} \gamma_{5}
		+ 2 \gamma_{2} \gamma_{6} 
		&= -\tfrac{5}{2}\eta \mu - \tfrac{17}{2}\mu,\\
		\label{eqn:k3:d} 2 \gamma_{0} \gamma_{3} + 2 \gamma_{1} \gamma_{4}
		+ 2 \gamma_{2} \gamma_{5} + 2 \gamma_{3} \gamma_{6}
		&= \tfrac{17}{2}\eta \mu + \tfrac{46}{3}\mu - \tfrac{1}{3},\\
		\label{eqn:k3:e} 2 \gamma_{0} \gamma_{2} + 2 \gamma_{1} \gamma_{3}
		+ 2 \gamma_{2} \gamma_{4} + 2 \gamma_{3} \gamma_{5}
		+ 2 \gamma_{4} \gamma_{6} &= - \tfrac{46}{3}\eta \mu
		+ \tfrac{1}{3}\eta - \tfrac{29}{2}\mu + \tfrac{3}{2},\\
		\label{eqn:k3:f} 2 \gamma_{0} \gamma_{1} + 2 \gamma_{1} \gamma_{2}
		+ 2 \gamma_{2} \gamma_{3} + 2 \gamma_{3} \gamma_{4}
		+ 2 \gamma_{4} \gamma_{5} + 2 \gamma_{5} \gamma_{6}
		&= \tfrac{29}{2}\eta \mu - \tfrac{10}{3}\eta + \tfrac{11}{2}\mu - 3,\\
		\label{eqn:k3:g} \gamma_{0}^{2} + \gamma_{1}^{2} + \gamma_{2}^{2}
		+ \gamma_{3}^{2} + \gamma_{4}^{2} + \gamma_{5}^{2} + \gamma_{6}^{2}
		&= - \tfrac{11}{2}\eta \mu + 3 \eta + \tfrac{11}{6}.
	\end{align}
\end{subequations}
\begin{lemma}
	Consider~\eqref{eqn:k3} and assume $\eta\geq 0$.
	Then,~\eqref{eqn:k3} is not solvable for $\mu > \tfrac{1}{7}$.
	For $\mu \leq \tfrac{1}{7}$, \eqref{eqn:k3} is solvable and there exist a positive definite matrix $G$ for this solution.
	Hence, \Cref{ass:summation:bdf,ass:summation:bdf:niceEigvals} are true for $\numBDF = \numDelay = 3$
	and $\omega\leq \tfrac{1}{7}$.
\end{lemma}
\begin{proof}
	Summing up the equations in~\eqref{eqn:k3} yields again $(\sum_{i=0}^6 \gamma_i)^2 = 0$ and $\sum_{i=0}^6 \gamma_i = 0$. 
	Now consider the sum of \cref{eqn:k3:b,eqn:k3:d,eqn:k3:f}, since these contain terms with product of $\gamma_i$'s with odd and even indices, which on factoring gives
	\begin{align*}
		2\, (\gamma_1 + \gamma_3 + \gamma_5)(\gamma_0 + \gamma_2 + \gamma_4 + \gamma_6) &= (\tfrac{1}{3} + \tfrac{17}{2} + \tfrac{29}{2})\eta\mu + (\tfrac{5}{2} + \tfrac{46}{3} + \tfrac{11}{2})\mu -\tfrac{10}{3}\eta + (-\tfrac{1}{3}-3)\nonumber\\
		&= \tfrac{70}{3}\eta\mu + \tfrac{70}{3}\mu -\tfrac{10}{3}\eta -\tfrac{10}{3}\nonumber\\
		&= \tfrac{10}{3}(1+\eta)(7\mu - 1).
	\end{align*}
	This yields
	\begin{align*}
		(\gamma_1+\gamma_3+\gamma_5)^2 
		= \tfrac{5}{3}\, (1+\eta)(1-7\mu),
	\end{align*}
	which establishes the first claim.
	For the second claim, we verify numerically that~\eqref{eqn:k3} is solvable; see~\Cref{fig:bdf3:resnormev} for the computed $\gamma_0,\ldots,\gamma_6$ and the residual norm for solving system~\eqref{eqn:k3} numerically. The maximum of the residual norms is several orders of magnitude smaller than the minimum eigenvalue, which is reported in \Cref{fig:bdfeigvals:k3}. Hence, the computed eigenvalues are numerically reliable and bounded away from zero, i.e., the matrix $G$ is positive definite.
\end{proof}

\begin{figure}[ht]
	\centering
	\ref*{legGammas3}\\[.3em]
	\begin{tikzpicture}
    \begin{axis}[
    width=2.5in,
    height=2.5in,
    xlabel={$\mu$},
    xmin=0, xmax=0.16,
    ymin=1.1e-16, ymax=5.0e-8,
    xtick={0.0, 0.05, 0.1, 0.15, 0.2},
    xticklabels={$0$, $0.05$, $0.1$, $0.15$, $0.2$},
    ymode={log},
    grid=major,
    cycle list name=summationsol,
    ]
    \addplot+ table[x=omega, y=resnorm] {df_bdf3_resnorms.dat};
    \addplot+[dashed, thick, black] coordinates {(1.0/7.0, 1.0e-17) (1.0/7.0, 1.0e-2)};
\end{axis}
\end{tikzpicture}
	\quad
	\begin{tikzpicture}
    \begin{axis}[
    width=2.5in,
    height=2.5in,
    xlabel={$\mu$},
    xmin=0, xmax=0.16,
    ymin=-1.3, ymax=1.2,
    xtick={0.0, 0.05, 0.1, 0.15, 0.2},
    xticklabels={$0$, $0.05$, $0.1$, $0.15$, $0.2$},
    grid=major,
    legend style={
      /tikz/every even column/.append style={column sep=0.5cm},
      font=\small,
    },
    legend columns=8,
    legend to name=legGammas3, 
    legend entries={$\gamma_0$, $\gamma_1$, $\gamma_2$, $\gamma_3$, $\gamma_4$, $\gamma_5$, $\gamma_6$, $\mu=\tfrac{1}{7}$},
    cycle list name=summationsol,
    ]
    \addplot+ table[x=omega, y=gamma_0] {df_bdf3_gammas.dat};
    \addplot+ table[x=omega, y=gamma_1] {df_bdf3_gammas.dat};
    \addplot+ table[x=omega, y=gamma_2] {df_bdf3_gammas.dat};
    \addplot+ table[x=omega, y=gamma_3] {df_bdf3_gammas.dat};
    \addplot+ table[x=omega, y=gamma_4] {df_bdf3_gammas.dat};
    \addplot+ table[x=omega, y=gamma_5] {df_bdf3_gammas.dat};
    \addplot+ table[x=omega, y=gamma_6] {df_bdf3_gammas.dat};
    \addplot+[dashed, thick, black] coordinates {(1.0/7.0, -2) (1.0/7.0, 2)};
\end{axis}
\end{tikzpicture}
	\caption{Residual norm (left) for solving the system for $\gamma$
	for \BDF-$3$ with $\eta=0.12$ and the corresponding solutions $\gamma$ (right).
	The dashed lines indicate the critical value $\omega=\tfrac{1}{7}$ from \Cref{th:mainResult}.}
	\label{fig:bdf3:resnormev}
\end{figure}
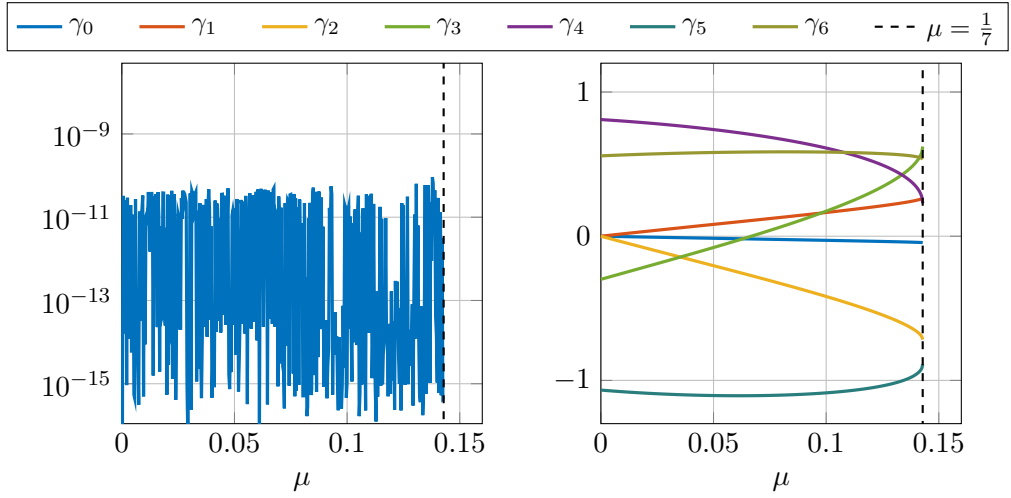

%
%
\subsection{Summary}

We summarize the previous subsections concerning the $\numBDF$-step schemes
(with $\numBDF = \numDelay$) in the upcoming \Cref{th:mainResult}.
In the formulation including $p$ and $u$, these schemes read for $\numBDF = 1$:
\begin{align*}
	\cA u^n - \cD^{*} p^{n-1} &= f^{n},\\
	\tfrac{1}{\tau}\cD (u^{n} - u^{n-1}) 
	+ \tfrac{1}{\tau}\cC (p^{n} - p^{n-1}) + \cB p^{n} &= g^{n},
\end{align*}
for $\numBDF = 2$:
\begin{align*}
	\cA u^n - \cD^{*} (2 p^{n-1} - p^{n-2}) &= f^{n},\\
	\tfrac{1}{\tau}\cD (\tfrac{3}{2}u^{n} - 2 u^{n-1} + \tfrac{1}{2}u^{n-2}) 
	 + \tfrac{1}{\tau}\cC (\tfrac{3}{2}p^{n} - 2 p^{n-1} + \tfrac{1}{2}p^{n-2})
	 + \cB p^{n} &= g^{n},
\end{align*}
and, for $\numBDF = 3$:
\begin{align*}
	\cA u^n - \cD^{*} (3 p^{n-1} - 3 p^{n-2} + p^{n-3})	&= f^{n},\\
	\tfrac{1}{\tau}\cD\, (\tfrac{11}{6}u^{n} - 3 u^{n-1}
	+ \tfrac{3}{2}u^{n-2} - \tfrac{1}{3}u^{n-3})\quad\qquad\qquad\qquad\qquad\qquad&\\
	+ \tfrac{1}{\tau}\cC\, (\tfrac{11}{6}p^{n} - 3 p^{n-1}
	+ \tfrac{3}{2}p^{n-2} - \tfrac{1}{3}p^{n-3})
	+ \cB p^{n} &= g^{n}.
\end{align*}

\begin{theorem}
\label{th:mainResult}
	Assume that one of the following applies:
	\begin{enumerate}[itemsep=0.3em]
		\item $\numBDF = 1$ and $\omega \leq 1$,
		\item $\numBDF = 2$ and $\omega \leq \tfrac{1}{3}$,
		\item $\numBDF = 3$ and $\omega \leq \tfrac{1}{7}$.
	\end{enumerate}
	Then \Cref{ass:summation:bdf,ass:summation:bdf:niceEigvals} are true for $\numDelay = \numBDF$. 
	Hence, for sufficiently smooth right-hand sides $f, g$ and consistent initial data, the solution of the original system~\eqref{eq:ppde} and its $k$-th order \BDF approximation given by~\eqref{eqn:par:opt:delay:bdf} satisfy 
	\begin{displaymath}
		\Vert u^{n} - {u}(t^{n}) \Vert^{2}_{\cV} + 
		\Vert p^{n} - {p}(t^{n}) \Vert^{2}_{\cHQ}
		\lesssim \tau^{2\numBDF}
		+ \sum_{\ell=0}^{k-1}\Vert p^{\ell} - {p}(t^{\ell})\Vert_{\cHQ}^{2}.
	\end{displaymath}
\end{theorem}
Note that the results are in line with the numerical
observations in~\cite{AltMU24}. 
A numerical verification of the derived convergence results is subject of the following section. 
%
%
%
\section{Numerical Experiment}
\label{sec:numerics}

On the unit square~$\Omega=(0,1)^2$ and $t\in[0,10]$, we set as exact solution 
\begin{align*}
	u(t,x,y) = -10 \mathrm{e}^{-\tfrac{5}{21} t}
	\begin{bmatrix}\cos(\pi x) \sin(\pi y)\\\sin(\pi x) \cos(\pi y)\end{bmatrix},
	\qquad p(t,x,y) = 10 \mathrm{e}^{-\tfrac{5}{21} t}\sin(\pi x) \sin(\pi y) 
\end{align*}
and define the right-hand sides accordingly. 
\begin{table}
	\centering
	\caption{Poroelasticity parameters for the numerical study in \Cref{sec:numerics}.}
	\label{tab:poroparam}
	\begin{tabular}{@{}c@{\qquad\quad}c@{\qquad\quad}c@{\qquad\quad}c@{\qquad\quad}c@{}}
		\toprule
		$\lambda$ & $\mu$ & $\tfrac{\kappa}{\nu}$ & $M$ & $\alpha$\\
		\midrule
		$0.5$ & $0.125$ & $0.05$ & $0.27$ & $0.5$\\
		\bottomrule
	\end{tabular}
\end{table}
The poroelasticity parameters are chosen as in \Cref{tab:poroparam},
which results in a coupling strength of $\omega \approx 0.11 < 1/7$.
A (uniform) mesh size of $2^{-7}$ is chosen for the spatial discretization.
To make the errors in time dominate, it is reasonable to consider
$(P_3,P_2)$ finite elements for $(u,p)$.
Note that the polynomial degree for $u$ is one higher than for $p$ as suggested in~\cite{ErnM09}.
For customary verification of the convergence result,
the errors from the monolithic and the corresponding decoupled methods at the final time are plotted in \Cref{fig:errorComparToy}, showing the
expected convergence rates.
\begin{figure}[ht]
	\hspace{-2em}
	\centering
	\begin{subfigure}[t]{.45\linewidth}
		\begin{tikzpicture}

  \begin{loglogaxis}[
    width=2.8in,
    height=2.8in,
    xmin=6.0e-02, xmax=2.0e-00,
    ymin=8.0e-7, ymax=5.0e-01,
    xtick={1.25, 0.625, 0.3125, 0.15625, 0.078125, 0.0390625},
    xticklabels={$2^{-3}$,$2^{-4}$,$2^{-5}$,$2^{-6}$,$2^{-7}$, $2^{-8}$},
    yticklabels={},
    xlabel={step size $\tau \times 10$},
    ylabel={$\Vert u(T) - u^{N} \Vert_{\cV}$},
    ylabel style={at={(-0.05,0.5)}},
    xmajorgrids,
    ymajorgrids,
    cycle list name=methodcompare,
    ]
    \addplot+
    table[x={tau}, y={EPsBDF1}] {u_H1.dat};
    \addplot+
    table[x={tau}, y={EPsBDF2}] {u_H1.dat};
    \addplot+
    table[x={tau}, y={EPsBDF3}] {u_H1.dat};
    \addplot+
    table[x={tau}, y={EPiBDF1}] {u_H1.dat};
    \addplot+
    table[x={tau}, y={EPiBDF2}] {u_H1.dat};
    \addplot+
    table[x={tau}, y={EPiBDF3}] {u_H1.dat};
    \addplot
    table[x={tau}, y={ORDER1}] {u_H1.dat};
    \addplot
    table[x={tau}, y={ORDER2}] {u_H1.dat};
    \addplot
    table[x={tau}, y={ORDER3}] {u_H1.dat};
  \end{loglogaxis}
\end{tikzpicture}
  
	\end{subfigure}
	\begin{subfigure}[t]{.45\linewidth}
		\begin{tikzpicture}

  \begin{loglogaxis}[
    width=2.8in,
    height=2.8in,
    xmin=6.0e-02, xmax=2.0e-00,
  	ymin=8.0e-7, ymax=5.0e-01,
    xtick={1.25, 0.625, 0.3125, 0.15625, 0.078125, 0.0390625},
    xticklabels={$2^{-3}$,$2^{-4}$,$2^{-5}$,$2^{-6}$,$2^{-7}$, $2^{-8}$},
    yticklabel style={at={(-0.7,0.5)}},
    xlabel={step size $\tau \times 10$},
	  ylabel={$\Vert p(T) - p^{N}\Vert_{\cHQ}$},
    ylabel style={at={(1.15,0.5)}},
    xmajorgrids,
    ymajorgrids,
    legend cell align={left},
    legend style={
      at={(0.5,-0.1)},
      anchor=north,
      /tikz/every even column/.append style={column sep=0.5cm},
      font=\small,
    },
    legend columns=3,
    legend to name=legToy, 
    cycle list name=methodcompare,
    ]
  \addplot+
  table[x={tau}, y={EPsBDF1}] {p_L2.dat};
  \addlegendentry{BDF-$1$ (semi-explicit)}
  \addplot+
  table[x={tau}, y={EPsBDF2}] {p_L2.dat};
  \addlegendentry{BDF-$2$ (semi-explicit)}
  \addplot+
  table[x={tau}, y={EPsBDF3}] {p_L2.dat};
  \addlegendentry{BDF-$3$ (semi-explicit)}
  \addplot+
  table[x={tau}, y={EPiBDF1}] {p_L2.dat};
  \addlegendentry{BDF-$1$}
  \addplot+
  table[x={tau}, y={EPiBDF2}] {p_L2.dat};
  \addlegendentry{BDF-$2$}
  \addplot+
  table[x={tau}, y={EPiBDF3}] {p_L2.dat};
  \addlegendentry{BDF-$3$}
  \addplot
  table[x={tau}, y={ORDER1}] {p_L2.dat};
  \addlegendentry{order 1}
  \addplot
  table[x={tau}, y={ORDER2}] {p_L2.dat};
  \addlegendentry{order 2}
  \addplot
  table[x={tau}, y={ORDER3}] {p_L2.dat};
  \addlegendentry{order 3}
\end{loglogaxis}
\end{tikzpicture}
	\end{subfigure}\\
	\ref*{legToy} 
	\caption{Convergence study for the numerical experiment on the unit square at time $t=T=10$ with mesh size~$h=2^{-7}$.
	The solid blue curves correspond to the semi-explicit method
	and the dashed red curves to the fully coupled methods (original \BDF schemes).
	Errors in $u$ (left) and $p$ (right).
	} 
	\label{fig:errorComparToy}
\end{figure}
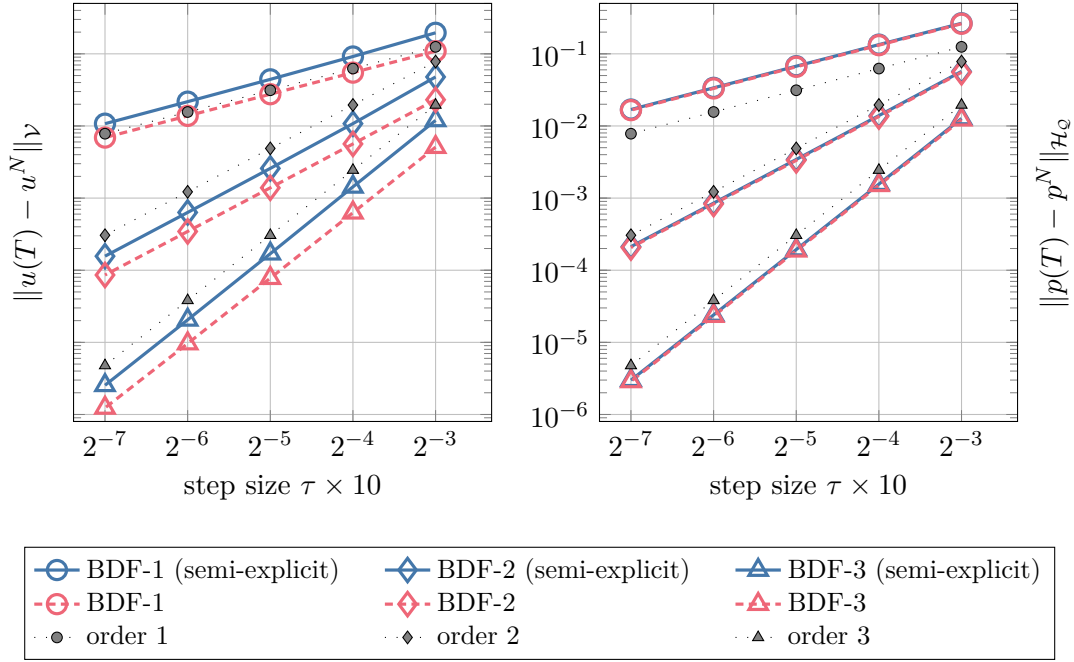
%
%
\section{Conclusions}
\label{sec:conclusion}
The convergence of semi-explicit methods based on \BDF-$k$ schemes is proven for $k\le 3$. The essential tool of the proof is the assumption on the G-stability of the schemes. In contrast to the standard G-stability concept, the inclusion of delay terms leads to a dependence on the spectrum of the involved operators of the problem. The evidence for this assumption is verified in this article. 

In future work, we aim to show that the semi-explicit schemes are A-stable for appropriate multipliers, which would directly imply G-stability. Moreover, we plan to extend the convergence proof to the semi-explicit methods based on Runge--Kutta schemes. 
%
%
\section*{Acknowledgements} 
This project is funded by the Deutsche Forschungsgemeinschaft (DFG, German Research
Foundation) - 467107679.
Parts of this work were carried out while RA and AM were affiliated with the
Institute of Mathematics, University of Augsburg.
RA was also affiliated with the Centre for Advanced Analytics and Predictive
Sciences (CAAPS).
Moreover, AM and BU acknowledge support by the Stuttgart Center for Simulation
Science (SimTech).
%
%
\bibliographystyle{plain-doi}
\bibliography{literature}
%
%
\appendix
\section{Auxiliary Results}
Within the analysis of the elliptic--parabolic system of \Cref{sec:prelim:delay} we use the following approximation property of the delay approach. 
\begin{lemma}
\label{lem:pTaylor}
Consider ${\bar p} \in C^{\numDelay+1}([-\numDelay\tau,T];\cHQ)$
and $\widehat{p}(\,\cdot\,;\tau) = \sum_{\ell=1}^{\numDelay}
c_{\numDelay,\ell}\shift{\ell\tau}{{\bar p}}$.
Then,
\begin{displaymath}
	{\bar p}(t) - {\widehat p}(t;\tau) 
	= R_T^{\numDelay, \numDelay}
	= \cO(\tau^{\numDelay}) \qquad\text{and}\qquad 
	{\dot {\bar p}}(t) - {\dot {\widehat p}}(t;\tau)
	= R_T^{\numDelay, \numDelay+1}
	= \cO(\tau^{\numDelay})
\end{displaymath}
with remainder $R_T^{\numDelay, k} = \frac{1}{(\numDelay-1)!}\sum_{\ell=1}^{\numDelay}(-1)^{(\ell-1)}\binom{\numDelay}{\ell}\int_{t-\ell\tau}^{t}(t -  \ell\tau - \zeta)^{\numDelay-1}{\bar p}^{(k)}(\zeta)\dzeta$. 
\end{lemma}
\begin{proof}
	The Taylor series expansion of $\shift{\ell\tau}{{\bar p}}$ around $t$ is
	\begin{align*}
		\shift{\ell\tau}{{\bar p}}(t) 
		= {\bar p}(t-\ell\tau)
		= \sum_{i=0}^{\numDelay-1}\frac{(-\ell\tau)^{i}}{i!}{\bar p}^{(i)}(t)
		- \frac{1}{(\numDelay-1)!}\int_{t-\ell\tau}^{t}(t -\ell\tau - \zeta)^{\numDelay-1}
		{\bar p}^{(\numDelay)}(\zeta)\dzeta.
	\end{align*}
	Substituting this in~\eqref{eq:p:approx:lagpol}, we obtain
	\begin{align}
		\label{eq:rem:term}
		{\widehat p}(t;\tau) 
		= \sum_{\ell=1}^{\numDelay}c_{\numDelay,\ell}\shift{\ell\tau}{{\bar p}}(t)
		= \sum_{i=0}^{\numDelay-1}
		\bigg(\sum_{\ell=1}^{\numDelay} \ell^{i}(-1)^{(\ell-1)}
		\binom{\numDelay}{\ell}\bigg)\frac{(-\tau)^{i}}{i!}{\bar p}^{(i)}(t)
		- R_T^{\numDelay, \numDelay},
	\end{align}
	where
	\begin{align}
		\label{eq:rem:1}
		R_T^{\numDelay, \numDelay} =
		\frac{1}{(\numDelay-1)!}\sum_{\ell=1}^{\numDelay}(-1)^{(\ell-1)}\binom{\numDelay}{\ell}\int_{t-\ell\tau}^{t}(t - \ell\tau - \zeta)^{\numDelay-1}{\bar p}^{(\numDelay)}(\zeta)\dzeta.
	\end{align}
	The first summand in~\eqref{eq:rem:term}, i.e., for $i=0$, we obtain
	\begin{displaymath}
		\sum_{\ell=1}^{\numDelay}\binom{\numDelay}{\ell}\, (-1)^{(\ell-1)} {\bar p}(t) 
		= \big(-(1-1)^\numDelay + 1\big)\, {\bar p}(t) 
		= {\bar p}(t).
	\end{displaymath}
	Observe that for $0<i<\numDelay$, we can obtain the inner sum in \eqref{eq:rem:term} by substituting $x=-1$ in 
	\begin{align*}
		\sum_{\ell=1}^{\numDelay}\binom{\numDelay}{\ell}\,\ell^i\, x^{(\ell-1)} &= \frac{1}{x}\Big[x \ddx (\,\cdot\,)\Big]^{i}\big((1+x)^\numDelay - 1\big) = \frac{1}{x}\Big[(1+x) \ddx (\,\cdot\,) - \ddx (\,\cdot\,)\Big]^{i}(1+x)^\numDelay\nonumber\\
		&=\frac{1}{x}\sum_{\ell=0}^{i}\binom{i}{\ell}\bigg(\Big[(1+x)\ddx(\,\cdot\,)\Big]^{i-\ell}(1+x)^{\numDelay}\bigg)\bigg(\Big[-\ddx(\,\cdot\,)\Big]^{\ell}(1+x)^{\numDelay}\bigg)\nonumber\\
		&=\frac{1}{x}\sum_{\ell=0}^{i}\binom{i}{\ell}\bigg(\numDelay^{(i-\ell)}(1+x)^{\numDelay}\bigg)\bigg((-1)^{\ell}\binom{\numDelay}{\ell}(1+x)^{\numDelay-\ell}\bigg), 
	\end{align*}
	which gives $0$ for all $0<i<\numDelay$. 
	Therefore, ${\bar p}(t) - {\widehat p}(t;\tau) = R_T^{\numDelay, \numDelay} = \cO(\tau^{\numDelay})$.
	Similarly, it can be shown that ${\dot {\bar p}}(t) - {\dot {\widehat p}}(t;\tau) = R_T^{\numDelay, \numDelay+1} = \cO(\tau^{\numDelay})$ with 
	\begin{align*}
		R_T^{\numDelay, \numDelay+1} = \frac{1}{(\numDelay-1)!}\sum_{\ell=1}^{\numDelay}(-1)^{(\ell-1)}\binom{\numDelay}{\ell}\int_{t-\ell\tau}^{t} (t - \ell\tau - \zeta)^{\numDelay-1}{\bar p}^{(\numDelay+1)}(\zeta)\dzeta.
		\qquad
		\qedhere
	\end{align*}
\end{proof}

\section{\texorpdfstring{$G$ matrices}{G matrices}}
In this section, we gather expressions of the $G$ matrix for the cases $\numBDF=2$ and $\numBDF=3$.

\subsection{\texorpdfstring{Case $\numBDF=\numDelay=2$}{Case k=d=2}}
\label{subsec:k2d2:G}
The $G$ matrix is given by
\begin{align*}
	G &= \left[\begin{matrix}
		\frac{3}{2} & -1 & \frac{1}{4} & 0\\
		-1 & \frac{3}{2} & -1 & \frac{1}{4}\\
		\frac{1}{4} & -1 & \frac{3}{2} & -1\\
		0 & \frac{1}{4} & -1 & \frac{3}{2}
	\end{matrix}\right]
	+ \mu \left[\begin{matrix}
		0 & \frac{3}{2} & - \frac{11}{4} & \frac{3}{2}\\
		\frac{3}{2} & 0 & \frac{3}{2} & - \frac{11}{4}\\
		- \frac{11}{4} & \frac{3}{2} & 0 & \frac{3}{2}\\
		\frac{3}{2} & - \frac{11}{4} & \frac{3}{2} & 0
	\end{matrix}\right] \nonumber\\
	&\quad+\eta\left[\begin{matrix}
		0 & - \frac{3}{4} & 0 & 0\\
		- \frac{3}{4} & 2 & -1 & 0\\
		0 & -1 & 2 & -1\\
		0 & 0 & -1 & 2
	\end{matrix}\right]
	+\eta\mu \left[\begin{matrix}
		0 & 0 & 0 & 0\\
		0 & -3 & \frac{11}{4} & - \frac{3}{2}\\
		0 & \frac{11}{4} & -3 & \frac{11}{4}\\
		0 & - \frac{3}{2} & \frac{11}{4} & -3
	\end{matrix}\right]
	- L L^{\top},
\end{align*}
where
\begin{equation*}
	L = \left[\begin{matrix}
		\gamma_{0} & 0 & 0 & 0\\
		\gamma_{1} & \gamma_{0} & 0 & 0\\
		\gamma_{2} & \gamma_{1} & \gamma_{0} & 0\\
		\gamma_{3} & \gamma_{2} & \gamma_{1} & \gamma_{0}
	\end{matrix}\right].
\end{equation*}

\subsection{\texorpdfstring{Case $\numBDF=\numDelay=3$}{Case k=d=3}}
The $G$ matrix is given by
\begin{align}
	\label{eqn:k3d3:G}
	G &= \left[\begin{matrix}
		\tfrac{11}{6} & - \tfrac{3}{2} & \tfrac{3}{4} & - \tfrac{1}{6} & 0 & 0\\
		- \tfrac{3}{2} & \tfrac{11}{6} & - \tfrac{3}{2} & \tfrac{3}{4} & - \tfrac{1}{6} & 0\\
		\tfrac{3}{4} & - \tfrac{3}{2} & \tfrac{11}{6} & - \tfrac{3}{2} & \tfrac{3}{4} & - \tfrac{1}{6}\\
		- \tfrac{1}{6} & \tfrac{3}{4} & - \tfrac{3}{2} & \tfrac{11}{6} & - \tfrac{3}{2} & \tfrac{3}{4}\\
		0 & - \tfrac{1}{6} & \tfrac{3}{4} & - \tfrac{3}{2} & \tfrac{11}{6} & - \tfrac{3}{2}\\
		0 & 0 & - \tfrac{1}{6} & \tfrac{3}{4} & - \tfrac{3}{2} & \tfrac{11}{6}
	\end{matrix}\right] 
	+\mu\left[\begin{matrix}
		0 & \tfrac{11}{4} & - \tfrac{29}{4} & \tfrac{23}{3} & - \tfrac{17}{4} & \tfrac{5}{4}\\
		\tfrac{11}{4} & 0 & \tfrac{11}{4} & - \tfrac{29}{4} & \tfrac{23}{3} & - \tfrac{17}{4}\\
		- \tfrac{29}{4} & \tfrac{11}{4} & 0 & \tfrac{11}{4} & - \tfrac{29}{4} & \tfrac{23}{3}\\
		\tfrac{23}{3} & - \tfrac{29}{4} & \tfrac{11}{4} & 0 & \tfrac{11}{4} & - \tfrac{29}{4}\\
		- \tfrac{17}{4} & \tfrac{23}{3} & - \tfrac{29}{4} & \tfrac{11}{4} & 0 & \tfrac{11}{4}\\
		\tfrac{5}{4} & - \tfrac{17}{4} & \tfrac{23}{3} & - \tfrac{29}{4} & \tfrac{11}{4} & 0
	\end{matrix}\right] \nonumber\\
	&\quad+ \eta \left[\begin{matrix}
		0 & - \tfrac{11}{12} & 0 & 0 & 0 & 0\\
		- \tfrac{11}{12} & 3 & - \tfrac{5}{3} & \tfrac{1}{6} & 0 & 0\\
		0 & - \tfrac{5}{3} & 3 & - \tfrac{5}{3} & \tfrac{1}{6} & 0\\
		0 & \tfrac{1}{6} & - \tfrac{5}{3} & 3 & - \tfrac{5}{3} & \tfrac{1}{6}\\
		0 & 0 & \tfrac{1}{6} & - \tfrac{5}{3} & 3 & - \tfrac{5}{3}\\
		0 & 0 & 0 & \tfrac{1}{6} & - \tfrac{5}{3} & 3
	\end{matrix}\right] 
	+\eta\mu \left[\begin{matrix}
		0 & 0 & 0 & 0 & 0 & 0\\
		0 & - \tfrac{11}{2} & \tfrac{29}{4} & - \tfrac{23}{3} & \tfrac{17}{4} & - \tfrac{5}{4}\\
		0 & \tfrac{29}{4} & - \tfrac{11}{2} & \tfrac{29}{4} & - \tfrac{23}{3} & \tfrac{17}{4}\\
		0 & - \tfrac{23}{3} & \tfrac{29}{4} & - \tfrac{11}{2} & \tfrac{29}{4} & - \tfrac{23}{3}\\
		0 & \tfrac{17}{4} & - \tfrac{23}{3} & \tfrac{29}{4} & - \tfrac{11}{2} & \tfrac{29}{4}\\
		0 & - \tfrac{5}{4} & \tfrac{17}{4} & - \tfrac{23}{3} & \tfrac{29}{4} & - \tfrac{11}{2}
	\end{matrix}\right]
	- L L^{\top},
\end{align}
where
\begin{equation*}
	L = \left[\begin{matrix}
		\gamma_{0} & 0 & 0 & 0 & 0 & 0\\
		\gamma_{1} & \gamma_{0} & 0 & 0 & 0 & 0\\
		\gamma_{2} & \gamma_{1} & \gamma_{0} & 0 & 0 & 0\\
		\gamma_{3} & \gamma_{2} & \gamma_{1} & \gamma_{0} & 0 & 0\\
		\gamma_{4} & \gamma_{3} & \gamma_{2} & \gamma_{1} & \gamma_{0} & 0\\
		\gamma_{5} & \gamma_{4} & \gamma_{3} & \gamma_{2} & \gamma_{1} & \gamma_{0}
	\end{matrix}\right].
\end{equation*}

\end{document}